\documentclass{amsart}
\usepackage{amssymb}
\usepackage{amsmath}
\usepackage{latexsym}
\usepackage{eepic}
\usepackage{epsfig}
\usepackage{graphicx}
\usepackage{pb-diagram,lamsarrow,pb-lams,amscd}
\usepackage{eufrak}
\usepackage{mathrsfs}
\usepackage[english]{babel}
\usepackage[all,cmtip]{xy}

\DeclareMathAlphabet{\mathpzc}{OT1}{pzc}{m}{it}
\newtheorem{theorem}{Theorem}[section]
\newtheorem{theorem-definition}[theorem]{Theorem-Definition}
\newtheorem{lemma-definition}[theorem]{Lemma-Definition}
\newtheorem{definition-prop}[theorem]{Proposition-Definition}

\newtheorem{prop}[theorem]{Proposition}
\newtheorem{lemma}[theorem]{Lemma}
\newtheorem{cor}[theorem]{Corollary}
\newtheorem{definition}[theorem]{Definition}

\theoremstyle{definition}

\newtheorem{question}[theorem]{Question}
\newtheorem{remark}[theorem]{Remark}

\newcommand{\LL}{\ensuremath{\mathbb{L}}}

\newcommand{\N}{\ensuremath{\mathbb{N}}}
\newcommand{\Z}{\ensuremath{\mathbb{Z}}}
\newcommand{\Q}{\ensuremath{\mathbb{Q}}}

\newcommand{\R}{\ensuremath{\mathbb{R}}}
\newcommand{\C}{\ensuremath{\mathbb{C}}}

\newcommand{\cT}{\ensuremath{\mathscr{T}}}
\newcommand{\cV}{\ensuremath{\mathscr{V}}}

\renewcommand{\R}{\ensuremath{\mathbb{R}}}
\renewcommand{\C}{\ensuremath{\mathbb{C}}}

\newcommand{\Spec}{\ensuremath{\mathrm{Spec}\,}}

\newcommand{\refl}{\ensuremath{\ast}}
\newcommand{\diag}{\ensuremath{\mathrm{Diag}}}
\newcommand{\Jord}{\ensuremath{\mathrm{Jord}}}
\newcommand{\ab}{\ensuremath{\mathrm{ab}}}
\newcommand{\tor}{\ensuremath{\mathrm{tor}}}
\newcommand{\pot}{\ensuremath{\mathrm{pot}}}

\newcommand{\Lie}{\ensuremath{\mathrm{Lie}}}
\newcommand{\Id}{\ensuremath{\mathrm{Id}}}
\newcommand{\et}{\ensuremath{\mathrm{et}}}
\newcommand{\ef}{\ensuremath{\mathrm{ef}}}
\newcommand{\an}{\ensuremath{\mathrm{an}}}
\newcommand{\Supp}{\ensuremath{\mathrm{Supp}}}
\newcommand{\dualm}{\ensuremath{\breve{m}}}
\newcommand{\Gr}{\ensuremath{\mathrm{Gr}}}

\numberwithin{equation}{section}
\begin{document}
\title{Jumps and monodromy of abelian varieties}
\author{Lars Halvard Halle}
\address{Institut f{\"u}r Algebraische Geometrie\\
Gottfried Wilhelm Leibniz Universit{\"a}t Hannover\\ Welfengarten
1
\\
30167 Hannover \\ Deutschland}
\curraddr{Matematisk Institutt\\
Universitetet i Oslo\\Postboks 1053
\\
Blindern\\0316 Oslo\\ Norway} \email{larshhal@math.uio.no}

\author[Johannes Nicaise]{Johannes Nicaise}
\address{KULeuven\\
Department of Mathematics\\ Celestijnenlaan 200B\\3001 Heverlee \\
Belgium} \email{johannes.nicaise@wis.kuleuven.be}

\thanks{The first author was partially supported by the Fund for Scientific
Research-Flanders (G.0318.06) and by DFG under grant Hu 337/6-1.
 The second author was partially supported by the Fund for Scientific Research - Flanders (G.0415.10) and ANR-06-BLAN-0183.
}

\begin{abstract}
We prove a strong form of the motivic monodromy conjecture for
 abelian varieties, by showing that the order of the unique pole
 of the motivic zeta function is equal to the maximal
rank of a Jordan block of the corresponding monodromy eigenvalue.
Moreover, we give a
 Hodge-theoretic interpretation of the fundamental invariants
 appearing in the proof.
 \\  MSC2000: 11G10, 14D05, 14D07
\end{abstract}
\maketitle

\section{Introduction}
Let $K$ be a henselian discretely valued field with algebraically
closed residue field $k$, and let $A$ be a tamely ramified abelian
$K$-variety of dimension $g$. In \cite{HaNi}, we introduced the
{\em motivic zeta function} $Z_A(T)$ of $A$. It is a formal power
series over the localized Grothendieck ring of $k$-varieties
$\mathcal{M}_k$, and it measures the behaviour of the N\'eron
model of $A$ under tame base change. We showed that
$Z_A(\LL^{-s})$ has a unique pole, which coincides with Chai's
{\em base change conductor} $c(A)$ of $A$, and that the order of
this pole equals $1+t_{\pot}(A)$, where $t_{\pot}(A)$ denotes the
{\em potential toric rank} of $A$. Moreover, we proved that for
every embedding of $\Q_\ell$ in $\C$, the value $\exp(2\pi c(A)i)$
is an eigenvalue of the tame monodromy transformation on the
$\ell$-adic cohomology of $A$ in degree $g$. The main ingredient
of the proof is Edixhoven's filtration on the special fiber of the
N\'eron model of $A$ \cite{edix}.

As we've explained in \cite{HaNi}, this result is a global version
of Denef and Loeser's motivic monodromy conjecture for
hypersurface singularities in characteristic zero. Denef and
Loeser's conjecture relates the poles of the motivic zeta function
of the singularity to monodromy eigenvalues on the nearby
cohomology. It is a motivic generalization of a conjecture of
Igusa's for the $p$-adic zeta function, which relates certain
arithmetic properties of polynomials $f$ in $\Z[x_1,\ldots,x_n]$
(namely, the asymptotic behaviour of the number of solutions of
the congruence $f\equiv 0$ modulo powers of a prime $p$) to the
structure of the singularities of the complex hypersurface defined
by the equation $f=0$. The conjectures of Igusa and Denef-Loeser
have been solved, for instance, in the case $n=2$
\cite{Loepadic}\cite{Rod}, but the general case remains wide open.
We refer to \cite{Ni-japan} for a survey.

There also exists a stronger form of Igusa's conjecture, which
says that the real parts of the poles of the $p$-adic zeta
function of $f$
 are roots of the Bernstein polynomial $b_f(s)$
of $f$, and that the order of each pole is at most the
multiplicity of the corresponding root of $b_f(s)$. This stronger
conjecture also has a motivic generalization, replacing the
$p$-adic zeta function by the motivic zeta function, and taking
for $f$ any polynomial with coefficients in a field of
characteristic zero (or, more generally, any regular function on a
smooth algebraic variety over a field of characteristic zero).

It is well-known that, for every complex polynomial $f$ and every
root $\alpha$ of $b_f(s)$, the value $\alpha'=\exp(2\pi i \alpha)$
is a monodromy eigenvalue on the nearby cohomology $R\psi_f(\C)_x$
of $f$ at some closed point $x$ of the zero locus $H_f$ of $f$
\cite{kashi2}\cite{Mal2}. Moreover, if $H_f$ has an isolated
singularity at $x$, then the multiplicity $m_{\alpha}$ of $\alpha$
as a root of the local Bernstein polynomial $b_{f,x}(s)$ of $f$ at
$x$ is closely related to the maximal size $m_{\alpha'}$ of the
Jordan blocks with eigenvalue $\alpha'$ of the monodromy
transformation on $R^{n-1}\psi_f(\C)_x$. In particular,
$m_{\alpha}\leq m_{\alpha'}$ if $\alpha\notin \Z$
\cite[7.1]{Mal2}.

The aim of the present paper is twofold. First, we prove a strong
form of the motivic monodromy conjecture for abelian varieties.
There is no good notion of Bernstein polynomial in our setting,
but we can look at the size of the Jordan blocks. We show that the
order $1+t_{\pot}(A)$ of the unique pole $s=c(A)$
 of the motivic zeta function $Z_A(\LL^{-s})$ is equal to the maximal rank of
 a
 Jordan block of the corresponding monodromy eigenvalue on the degree $g$ cohomology of $A$ (Theorem \ref{thm-blocksize}).
 Next, we use the
 theory of N\'eron models of variations of
 Hodge structures to give a
 Hodge-theoretic interpretation of the jumps in Edixhoven's filtration.
 This is done in Theorems \ref{theo-mhs} and \ref{thm-weight}.
 In
 \cite[2.7]{HaNi}, we speculated on a generalization of the monodromy
 conjecture to Calabi-Yau varieties over $\C((t))$ (i.e., smooth, proper,
 geometrically connected $\C((t))$-varieties with trivial canonical sheaf); we hope that the translation
 of Edixhoven's invariants to Hodge theory will help to extend the
 proof of the monodromy conjecture to that setting.

 In
 order to obtain these results, we divide Edixhoven's jumps into three
 types: toric, abelian, and dual abelian. The basic properties of these
 types are discussed in Section \ref{sec-tormult}. Not all of these results are
 used in the proofs of the main results of the paper. We include
 them for the sake of completeness and because we believe that they are of independent interest.  The reader who is
 only interested in Theorems \ref{thm-jord},
 \ref{thm-blocksize}, \ref{theo-mhs} and \ref{thm-weight} may skip Lemma \ref{lemma-exact}, Proposition \ref{prop-dualeq} and all the results in Section \ref{sec-tormult} after
 Proposition \ref{prop-jumps}.

 The different types of jumps are related to the monodromy
 transformation on the Tate module of $A$ in Section
 \ref{sec-jumpsmon}. Toric jumps correspond to monodromy
 eigenvalues with Jordan block of size two, and the abelian and dual
 abelian jumps to monodromy eigenvalues with Jordan block of size
 one (Theorem \ref{thm-jord}). If $K=\C((t))$, then the abelian
 and dual abelian jumps can be distinguished by looking at the
 Hodge type in the limit mixed Hodge structure (Theorem \ref{thm-weight}).

\subsection*{Acknowledgements} We are grateful to the referee for
carefully reading the manuscript and suggesting several
improvements to the paper. The second author is indebted to
Antoine Chambert-Loir  for helpful suggestions concerning Section
\ref{subsec-weight}.

\section{Preliminaries and notation}
We denote by $R$ a Henselian discrete valuation ring, with
quotient field $K$ and algebraically closed residue field $k$. We
denote by $K^a$ an algebraic closure of $K$, by $K^s$ the
separable closure of $K$ in $K^a$, and by $K^t$ the tame closure
of $K$ in $K^s$. We fix a topological generator $\sigma$ of the
tame monodromy group $G(K^t/K)$. We denote by $p$ the
characteristic exponent of $k$, and by $\N'$ the set of integers
$d>0$ prime to $p$. We denote by
$$(\cdot)_s:(R-\mathrm{Schemes})\to
(k-\mathrm{Schemes}):\mathcal{X}\mapsto
\mathcal{X}_s=\mathcal{X}\times_R k$$ the special fiber functor.

For every abelian variety $B$ over a field $F$, we denote its dual
abelian variety by $B^{\vee}$. For every abelian $K$-variety $A$
with N\'eron model $\mathcal{A}$, we denote by $t(A)$, $u(A)$ and
$a(A)$ the reductive, resp. unipotent, resp. abelian rank of
$\mathcal{A}_s^o$. We call these values the {\em toric}, resp.
{\em unipotent}, resp. {\em abelian} rank of $A$. Obviously, their
sum equals the dimension of $A$.

By Grothendieck's semi-stable reduction theorem, there exists a
 finite extension $K'$ of $K$ in $K^s$ such that $A\times_K
K'$ has semi-abelian reduction \cite[IX.3.6]{sga7a}. This means
that the identity component of the special fiber of its N\'eron
model is a semi-abelian $k$-variety; equivalently, $u(A\times_K
K')=0$. The value $t_{\pot}(A)=t(A\times_K K')$ is called the {\em
potential toric rank} of $A$, and the value
$a_{\pot}(A)=a(A\times_K K')$ the {\em potential abelian rank}. It
follows from \cite[IX.3.9]{sga7a} that these values are
independent of the choice of $K'$.
 We say that $A$ is {\em tamely
ramified} if we can take for $K'$ a {\em tame} finite extension of
$K$ (since $k$ is algebraically closed, this means that the degree
$[K':K]$ is prime to $p$) .

For every scheme $S$, every $S$-group scheme $G$ and every integer
$n>0$, we denote by $n_G:G\rightarrow G$ the multiplication by
$n$, and by $_{n}G$ its kernel.

If $\mathcal{S}$ is a set, and $g:\mathcal{S}\rightarrow \R$ a
function with finite support, we set
$$\|g\|=\sum_{s\in \mathcal{S}}g(s).$$
We denote the support of $g$ by $\Supp(g)$.

\begin{definition}
For every function
$$f:\Q/\Z\rightarrow \R$$ we define its {\em reflection} $$f^{\refl}:\Q/\Z\rightarrow \R$$ by
$$f^{\refl}(x)=f(-x).$$

We call $f$ {\em complete} if for every $x\in \Q/\Z$, the value
$f(x)$ only depends on the order of $x$ in the group $\Q/\Z$. We
say that $f$ is {\em semi-complete} if $f+f^{\refl}$ is complete.
\end{definition}

Consider a function
$$f:\Q/\Z\rightarrow \N$$ and assume that there exists an element
$e$ of $\Z_{>0}$ such that $\Supp(f)$ is contained in
$((1/e)\Z)/\Z$. Let $F$ be any algebraically closed field such
that $e$ is prime to the characteristic exponent $p'$ of $F$. For
each generator $\zeta$ of $\mu_e(F)$, we put
$$P_{f,\zeta}(t)=\prod_{i\in ((1/e)\Z)/\Z}(t-\zeta^{i\cdot
e})^{f(i)}$$ in $F[t]$. For each integer $d>0$, we denote by
$\Phi_d(t)$ the cyclotomic polynomial whose roots are the
primitive $d$-th roots of unity. We say that $\Phi_d(t)$ is
$F$-tame if $d$ is prime to $p'$.

\begin{lemma}\label{lemm-comp}
The function $f$ is complete if and only if for some generator
$\zeta$ of $\mu_e(F)$, the polynomial $P_{f,\zeta}(t)$ is the
image of a product $Q_f(t)$ of $F$-tame cyclotomic polynomials
under the unique ring morphism
$$\rho:\Z[t]\rightarrow F[t]$$
mapping $t$ to $t$. If $f$ is complete, then $P_{f,\zeta}(t)$ is
independent of $\zeta$ and $e$, and $Q_f(t)$ is unique. In that
case, if we choose a primitive $e$-th root of unity $\xi$ in an
algebraic closure $\Q^a$ of $\Q$, then
$$Q_f(t)=\prod_{i\in ((1/e)\Z)/\Z}(t-\xi^{i\cdot
e})^{f(i)}.$$
\end{lemma}
\begin{proof}
First, assume that $f$ is complete, and put $$Q_f(t)=\prod_{i\in
((1/e)\Z)/\Z}(t-\xi^{i\cdot e})^{f(i)}$$ for some primitive $e$-th
root of unity $\xi$ in $\Q^a$. Then $Q_f(t)$ is a product of
$F$-tame cyclotomic polynomials, because $f$ is complete and $e$
is prime to $p'$. There is a unique ring morphism
$$\widetilde{\rho}:\Z[\xi][t]\rightarrow F[t]$$ that maps $\xi$ to $\zeta$ and $t$
to $t$.
 We clearly have $\widetilde{\rho}(Q_f(t))=P_{f,\zeta}(T)$. Since $Q_f(t)$ belongs to $\Z[t]$, it follows
  that $\rho(Q_f(t))=P_{f,\zeta}(T)$ so that
$P_{f,\zeta}(t)$ does not depend on $\zeta$. Uniqueness of
$Q_f(t)$ follows from \cite[5.10]{HaNi}.

Conversely, if $P_{f,\zeta}(t)$ is the image under $\rho$ of a
product $Q(t)$ of $F$-tame cyclotomic polynomials, then it is
easily seen that $f$ is complete.
\end{proof}
\section{Toric and abelian multiplicity}\label{sec-tormult}
\subsection{Galois action on N\'eron models}\label{subsec-neron}
Let $A$ be a tamely ramified abelian $K$-variety of dimension $g$,
and let $K'$ be a finite extension of $K$ in $K^t$ such that
$A'=A\times_K K'$ has semi-abelian reduction. We denote by $R'$
the integral closure of $R$ in $K'$, and by $\mathfrak{m}'$ the
maximal ideal of $R'$. We put $d=[K':K]$.

We denote by $\mu$ the Galois group  $G(K'/K)$, and we let $\mu$
act on $K'$ from the left. The action of $\zeta\in \mu$ on
$\mathfrak{m}'/(\mathfrak{m}')^2$ is multiplication by
$\iota(\zeta)$, for some element $\iota(\zeta)$ in the group
$\mu_d(k)$ of $d$-th roots of unity in $k$, and the map
$$\mu\to \mu_d(k):\zeta\mapsto \iota(\zeta)$$ is an isomorphism.

We denote by $\mathcal{A}$ and $\mathcal{A}'$ the N\'eron models
of $A$, resp. $A'$. By the universal property of the N\'eron
model, there exists a unique morphism of $R'$-group schemes
$$h:\mathcal{A}\times_R R'\to \mathcal{A}'$$
that extends the canonical isomorphism between the generic fibers.
It induces an injective morphism of free rank $g$ $R'$-modules
$$\Lie(h):\Lie(\mathcal{A}\times_R R')\to \Lie(\mathcal{A}').$$

\begin{definition}[Chai \cite{chai}]
 The {\em base change conductor} $c(A)$ of $A$
is $[K':K]^{-1}$ times the length of the cokernel of $\Lie(h)$.
\end{definition}

The definition does not require that $A$ is tamely ramified. The
base change conductor is a positive rational number, independent
of the choice of $K'$. It vanishes if and only if $A$ has
semi-abelian reduction.

The right $\mu$-action on $A'$ extends uniquely to a right
$\mu$-action on $\mathcal{A}'$ such that the structural morphism
$$\mathcal{A}'\to\Spec R'$$ is $\mu$-equivariant. We denote by
\begin{equation}\label{eq-cheval}
0\rightarrow T\rightarrow (\mathcal{A}'_s)^o\rightarrow
B\rightarrow 0\end{equation} the Chevalley decomposition of
$(\mathcal{A}'_s)^o$, with $T$ a $k$-torus and $B$ an abelian
$k$-variety. There exist unique right $\mu$-actions on $T$, resp.
$B$, such that (\ref{eq-cheval}) is $\mu$-equivariant. The right
$\mu$-action on $B$ induces a left $\mu$-action on the dual
abelian variety $B^{\vee}$.

\begin{lemma}\label{lemma-exact}
(1) The complex
\begin{equation}\label{eq-invar}
0\rightarrow T^{\mu}\rightarrow
((\mathcal{A}'_s)^o)^{\mu}\rightarrow B^{\mu},
\end{equation}
obtained from (\ref{eq-cheval}) by taking $\mu$-invariants, is an
exact complex of smooth group schemes over $k$.

 Taking identity components, we get a complex
\begin{equation}\label{eq-invar-id}
0\rightarrow (T^{\mu})^o\rightarrow
((\mathcal{A}'_s)^\mu)^{o}\rightarrow (B^{\mu})^o\rightarrow 0
\end{equation}
of smooth group schemes over $k$ that is exact at the left and at
the right. The quotient
$$B'=((\mathcal{A}'_s)^\mu)^{o}/(T^{\mu})^o$$ is an abelian
$k$-variety, and the natural morphism $f:B'\rightarrow
(B^{\mu})^o$ is a separable isogeny.

(2) If we denote by $h$ the unique morphism
$$h:\mathcal{A}\times_R R'\rightarrow \mathcal{A}'$$ extending the
natural isomorphism between the generic fibers, then the
$k$-morphism $h_s:\mathcal{A}_s\rightarrow \mathcal{A}'_s$ factors
through a morphism
$$g:\mathcal{A}_s\rightarrow (\mathcal{A}'_s)^\mu.$$
The morphism $g$ is smooth and surjective, and its kernel is a
connected unipotent smooth algebraic $k$-group. The identity
component $((\mathcal{A}'_s)^\mu)^o$ is semi-abelian.
\end{lemma}
\begin{proof}
(1) It follows from \cite[3.4]{edix} that the group schemes in
\eqref{eq-invar} are smooth over $k$. Exactness of
(\ref{eq-invar}) is clear. The morphism
$$\alpha:(\mathcal{A}'_s)^o\rightarrow B$$  is smooth, since $T$
is smooth over $k$ \cite[VI$_\mathrm{B}$.9.2]{sga3.1}. It follows
from \cite[3.5]{edix} that
$$\alpha^{\mu}:((\mathcal{A}'_s)^o)^{\mu}\rightarrow B^{\mu}$$ is
smooth, as well.

Taking identity components in (\ref{eq-invar}), we get a  complex
$$\begin{CD}
(T^{\mu})^o@>\beta>>
((\mathcal{A}'_s)^\mu)^{o}@>\gamma=(\alpha^{\mu})^o>> (B^{\mu})^o
\end{CD}$$
 of smooth group schemes over $k$.
 It is obvious that $\beta$ is a monomorphism, and thus a closed immersion \cite[VI$_B$.1.4.2]{sga3.1}. Surjectivity of $\gamma$
follows from \cite[VI$_\mathrm{B}$.3.11]{sga3.1}, since $\gamma$
is smooth, and thus flat. We put
$$B'=((\mathcal{A}'_s)^\mu)^{o}/(T^{\mu})^o$$
This is a connected smooth algebraic $k$-group, by
\cite[VI$_\mathrm{A}$.3.2 and VI$_\mathrm{B}$.9.2]{sga3.1}.

Consider the natural morphism $f:B'\rightarrow (B^{\mu})^o$. It is
surjective, because $\gamma$ is surjective. The dimension of $B'$
is equal to $$\dim\,(\mathcal{A}'_s)^\mu - \dim\,T^{\mu},$$ which
is at most $\dim\,B^{\mu}$ by exactness of \eqref{eq-invar}.
Surjectivity of $f$ then implies that $B'$ and $B^{\mu}$ must have
the same dimension, so that $f$ has finite kernel. Thus $f$ is an
isogeny and $B'$ is an abelian variety. The kernel of $f$ is
canonically isomorphic to
$$\ker(\gamma)/(T^{\mu})^o.$$
Since $\gamma$ is smooth, we know that $\ker(\gamma)$ is smooth
over $k$, so that $\ker(f)$ is smooth over $k$, by
\cite[VI$_\mathrm{B}$.9.2]{sga3.1}. Hence, $f$ is a separable
isogeny.

%
%
%

(2)  Since $h$ is $\mu$-equivariant, and $\mu$ acts trivially on
the special fiber $\mathcal{A}_s$ of $\mathcal{A}\times_R R'$, the
morphism $h_s$ factors through a morphism
$g:\mathcal{A}_s\rightarrow (\mathcal{A}'_s)^{\mu}$. By
\cite[5.3]{edix}, the morphism $g$ is smooth and surjective, and
its kernel is a connected unipotent smooth algebraic $k$-group. By
\cite[3.4]{edix}, $((\mathcal{A}'_s)^{\mu})^o$ is a connected
smooth closed $k$-subgroup scheme of the semi-abelian $k$-group
scheme $(\mathcal{A}'_s)^o$, so that $((\mathcal{A}'_s)^{\mu})^o$
is semi-abelian by \cite[5.2]{HaNi}.
\end{proof}

\subsection{Multiplicity functions}\label{subsec-multipl}
Fix an element $e\in \N'$. For every finite dimensional $k$-vector
space $V$ with a right $\mu_e(k)$-action
$$*: V\times \mu_e(k) \rightarrow V:(v,\zeta)\mapsto v*\zeta$$
 and for
every integer $i$ in $\{0,\ldots,e-1\}$, we denote by $V[i]$ the
maximal subspace of $V$ such that
$$v*\zeta=\zeta^i\cdot v$$ for all $\zeta\in \mu_e(k)$
and all $v\in V[i]$. We define the {\em multiplicity function}
$$m_{V,\mu_e(k)}:\Q/\Z\rightarrow \N$$ by
$$\left\{\begin{array}{rcll}m_{V,\mu_e(k)}(i/e)&=&\mathrm{dim}(V[i])&\mbox{ for }i\in
\{0,\ldots,e-1\} \\ m_{V,\mu_e(k)}(x)&=&0&\mbox{ if }x\notin
((1/e)\Z)/\Z\end{array}\right.$$
 Note that $m_{V,\mu_e(k)}$ determines the
$k[\mu_e(k)]$-module $V$ up to isomorphism, since the order of
$\mu_e(k)$ is invertible in $k$.

In an analogous way, we define the multiplicity function
$m_{\mu_e(k),W}$ for a finite dimensional $k$-vector space $W$
 with left $\mu_e(k)$-action.  The inverse of the right $\mu_e(k)$-action
on $V$ is the left action
$$\mu_e(k)\times V \rightarrow V:(\zeta,v)\mapsto v*\zeta^{-1}.$$
 Its multiplicity function $m_{\mu_e(k),V}$ is equal to the reflection
$(m_{V,\mu_e(k)})^{\refl}$ of the multiplicity function
$m_{V,\mu_e(k)}$.

Let $A$ be a tamely ramified abelian $K$-variety. We adopt the
notations of Section \ref{subsec-neron}. In the set-up of
\eqref{eq-cheval}, the group $\mu\cong \mu_d(k)$ acts on the
$k$-vector spaces $\Lie(T)$, $\Lie(\mathcal{A}_s')$ and $\Lie(B)$
from the right, and on $\Lie(B^{\vee})$ from the left (via the
dual action of $\mu$ on $B^{\vee}$). Hence, we can state the
following definitions.
\begin{definition}\label{def-mult}
We define the toric multiplicity function $m_A^{\tor}$ of $A$ by
$$m_A^{\tor}=m_{\Lie(T),\mu}.$$

We define the abelian multiplicity function $m_{A}^{\ab}$ of $A$
by
$$m_A^{\ab}=m_{\Lie(B),\mu}.$$

We define the dual abelian multiplicity function
$\dualm_{A}^{\ab}$ of $A$ by
$$\dualm_{A}^{\ab}=m_{\mu,\Lie(B^{\vee})}.$$

Finally, we define the multiplicity function $m_A$ of $A$ by
$$m_A=m_A^{\tor}+m_A^{\ab}=m_{\Lie(\mathcal{A}'_s),\mu}.$$
\end{definition}
Using \cite[IX.3.9]{sga7a}, it is easily checked that these
definitions only depend on $A$, and not on the choice of $K'$.

\begin{prop}\label{prop-dualeq}
For every tamely ramified abelian $K$-variety $A$, we have
$$\dualm^{\ab}_A=(m_{A^{\vee}}^{\ab})^{\refl}.$$
\end{prop}
\begin{proof}
We adopt the notations of Section \ref{subsec-neron}. We set
$(A^{\vee})'=A^{\vee}\times_K K'$ and we denote its N\'eron model
by $(\mathcal{A}^{\vee})'$. The identity component of
$(\mathcal{A}^{\vee})'_s$ is a semi-abelian $k$-variety
\cite[IX.2.2.7]{sga7a}. We denote by $C$ its abelian part.

As explained in Section \ref{subsec-neron}, the left Galois action
of $\mu$ on $K'$ induces a right action of $\mu$ on $C$. The
canonical divisorial correspondence on $A\times_K A^{\vee}$
induces a divisorial correspondence on $B\times_k C$ that
identifies $C$ with the dual abelian variety of $B$
\cite[IX.5.4]{sga7a}. It suffices to show that the right
$\mu$-action on $C$ is the inverse of the dual of the right
$\mu$-action on $B$. To this end, we take a closer look at the
construction of the divisorial correspondence on $B\times_k C$.
Here we need the language of biextensions \cite[VII and
VIII]{sga7a}. We note that the following proof does not use the
assumption that $A$ is tamely ramified and that $K'$ is a tame
extension of $K$.

The canonical divisorial correspondence on $A\times_K A^{\vee}$
can be interpreted as a {\em Poincar\'e biextension} $\mathscr{P}$
of $(A,A^{\vee})$ by $\mathbb{G}_{m,K}$ \cite[VII.2.9.5]{sga7a},
which is defined up to isomorphism. It induces a biextension
$\mathscr{P}'$ of $(A',(A^{\vee})')$ by $\mathbb{G}_{m,K'}$ by
base change. By \cite[VIII.7.1]{sga7a}, the biextension
$\mathscr{P}'$ extends uniquely to a biextension of
$((\mathcal{A}')^o,((\mathcal{A}^{\vee})')^o)$ by
$\mathbb{G}_{m,R'}$, which restricts to a biextension
$\mathscr{P}'_s$ of
$((\mathcal{A}')^o_s,((\mathcal{A}^{\vee})')^o_s)$ by
$\mathbb{G}_{m,k}$. By \cite[VIII.4.8]{sga7a}, $\mathscr{P}'_s$
induces a biextension $\mathscr{Q}$ of $(B,C)$ by
$\mathbb{G}_{m,k}$ that is characterized (up to isomorphism) by
the fact that its pullback to
$((\mathcal{A}')^o_s,((\mathcal{A}^{\vee})')^o_s)$ is isomorphic
to $\mathscr{P}'_s$. The theorem in \cite[IX.5.4]{sga7a} asserts
that $\mathscr{Q}$ is a Poincar\'e biextension.

 For every element $\zeta$ of $\mu$, we denote by $r_{\zeta}$  the right
 multiplication by $\zeta$ on $B$ and $C$.
Since $\mathscr{P}'$ is obtained from the biextension
$\mathscr{P}$ over $K$ by base change to $K'$, it follows easily
from the construction that the pullback of the biextension
$\mathscr{Q}$ through the $k$-morphisms
\begin{eqnarray*}r_\zeta&:&B\to B\\ r_\zeta&:&C\to C\end{eqnarray*} is
isomorphic to $\mathscr{Q}$.   Interpreting $\mathscr{Q}$ as an
 isomorphism $$i:B\to C^{\vee}$$ in the way of \cite[VIII.3.2.2]{sga7a}, this means
 that the diagram
 $$\begin{CD}
B@>i>> C^{\vee}
\\ @V r_\zeta VV @AA(r_\zeta)^{\vee}A
\\ B@>i>> C^{\vee}
 \end{CD}$$ commutes, which is what we wanted to show.
\end{proof}



In the following proposition, we see how the multiplicity
functions of a tamely ramified abelian $K$-variety $A$ are related
to Edixhoven's jumps and Chai's elementary divisors of $A$. These
jumps and elementary divisors are rational numbers in $[0,1[$ that
measure the behaviour of the N\'eron model of $A$ under tame
ramification of the base field $K$. For the definition of
Edixhoven's jumps, we refer to \cite[5.4.5]{edix}. The terminology
we use is the one from \cite[4.12]{HaNi}; in particular, the
multiplicity of a jump is defined there. For Chai's elementary
divisors, we refer to \cite[2.4]{chai}. By definition, the base
change conductor $c(A)$ of $A$ is equal to the sum of the
elementary divisors.

\begin{prop}\label{prop-jumps} Let $A$ be a tamely ramified
abelian $K$-variety. The functions $m_A$, $m_A^{\tor}$,
$m_A^{\ab}$ and $\dualm^{\ab}_A$
  are supported on
$$((1/e)\Z)/\Z,$$ with $e$ the degree of the minimal
extension of $K$ where $A$ acquires semi-abelian reduction.

If we identify $[0,1[\,\cap \Q$ with $\Q/\Z$  via the bijection
$$[0,1[\,\cap \Q\rightarrow \Q/\Z:x\mapsto x\ \mathrm{mod}\ \Z$$ then for every $x\in [0,1[\,\cap
\Q$,
 the value $m_A(x)$ is
equal to the multiplicity of $x$ as a jump in Edixhoven's
filtration for $A$. In particular, the support of $m_A$ is the set
of jumps in Edixhoven's filtration. The value $m_A(x)$ is also
equal to the number of Chai's elementary divisors of $A$ that are
equal to $x$, and the base change conductor $c(A)$ of $A$ is given
by
$$c(A)=\sum_{x\in [0,1[\,\cap \Q}(m_A(x)\cdot x) .$$
\end{prop}
\begin{proof}
See \cite[5.3 and 5.4.5]{edix} and \cite[4.8 and 4.13 and
4.18]{HaNi}.
\end{proof}

\begin{prop}\label{prop-basic}
We have the following equalities:
$$\begin{array}{lclclclcl} \|m_A\|&=&\mathrm{dim}(A),& & \|m_A^{\ab}\|&=&\|\dualm_A^{\ab}\|
&=&a_{\pot}(A),
\\[1,5ex] \|m_A^{\tor}\|&=&t_{\pot}(A),& &m^{\ab}_A(0)&=&\dualm^{\ab}_A(0)&=&a(A),
\\[1,5ex] m^{\tor}_A(0)&=& t(A). &&&&&&
\end{array}$$
Moreover, we have $$ \sum_{x\in (\Q/\Z)\setminus
\{0\}}m_A(x)=u(A).$$
\end{prop}
\begin{proof}
We adopt the notations of Section \ref{subsec-neron}. It follows
immediately from the definitions that
$$\begin{array}{rcccccl}
& &\|m_A\|&=&\mathrm{dim}(\Lie(\mathcal{A}'_s))&=&\mathrm{dim}(A),
\\ & &\|m_A^{\tor}\|&=&\mathrm{dim}(\Lie(T))&=&t_{\pot}(A),
\\
\|\dualm_A^{\ab}\|&=&\|m_A^{\ab}\|&=&\mathrm{dim}(\Lie(B))&=&a_{\pot}(A).
\end{array}$$
By Lemma \ref{lemma-exact}, the abelian, resp., reductive rank of
$\mathcal{A}_s^o$ is equal to the abelian, resp., reductive rank
of the semi-abelian $k$-variety $((\mathcal{A}'_s)^{\mu})^o$. In
the notation of Lemma \ref{lemma-exact}, the Chevalley
decomposition of $((\mathcal{A}'_s)^{\mu})^o$ is given by
$$0\rightarrow (T^{\mu})^o\rightarrow
((\mathcal{A}'_s)^{\mu})^o\rightarrow B'\rightarrow 0$$ and there
exists a  separable isogeny $f:B'\rightarrow (B^{\mu})^o$. By
\cite[3.2]{edix}, the natural morphisms $$\begin{array}{c}
\Lie(T^{\mu})\rightarrow \Lie(T)^{\mu}=\Lie(T)[0]
\\ \Lie(B^{\mu})\rightarrow \Lie(B)^{\mu}=\Lie(B)[0]
\end{array}$$
are isomorphisms. Since $\Lie(f)$ is also an isomorphism, we find
\begin{eqnarray*}
m^{\tor}_A(0)&=& t(A),
\\ m^{\ab}_A(0)&=&a(A).
\end{eqnarray*}
 It follows that
\begin{eqnarray*}
\sum_{x\in (\Q/\Z)\setminus
\{0\}}m_A(x)&=&\|m_A\|-m_A^{\tor}(0)-m_A^{\ab}(0)
\\ &=&\mathrm{dim}(A)-t(A)-a(A)
\\&=&u(A).
\end{eqnarray*}
By Proposition \ref{prop-dualeq}, we have
$$\dualm^{\ab}_A(0)=m^{\ab}_{A^{\vee}}(0)$$ and we've just seen
that this value is equal to the abelian rank $a(A^{\vee})$ of
$A^{\vee}$. But the abelian ranks of $A$ and $A^{\vee}$ coincide,
by \cite[2.2.7]{sga7a}, so that
$$\dualm^{\ab}_A(0)=a(A).$$
\end{proof}

\begin{lemma}\label{lemm-prod}
If $A_1$ and $A_2$ are tamely ramified abelian $K$-varieties, then
\begin{eqnarray*}
m_{A_1\times_K A_2}^{\tor}&=&m_{A_1}^{\tor}+m_{A_2}^{\tor},
\\ m_{A_1\times_K A_2}^{\ab}&=&m_{A_1}^{\ab}+m_{A_2}^{\ab},
\\ \dualm_{A_1\times_K
A_2}^{\ab}&=&\dualm_{A_1}^{\ab}+\dualm_{A_2}^{\ab}.
\end{eqnarray*}
\end{lemma}
\begin{proof}
If we denote by $\mathcal{A}_1$ and $\mathcal{A}_2$ the N\'eron
models of $A_1$, resp. $A_2$, then it follows immediately from the
universal property of the N\'eron model that
$\mathcal{A}_1\times_R \mathcal{A}_2$ is a N\'eron model for
$A_1\times_K A_2$. Since the Chevalley decomposition of a
connected smooth algebraic $k$-group commutes with finite fibered
products over $k$ and $\Lie(G_1\times_k G_2)$ is canonically
isomorphic to $\Lie(G_1)\oplus \Lie(G_2)$ for any pair of
algebraic $k$-groups $G_1$, $G_2$, the result follows.
\end{proof}

\begin{prop} Let $A$ be a tamely ramified abelian $K$-variety.
Let $L$ be a finite tame extension of $K$ of degree $e$, and put
$A_L=A\times_K L$. Then for each $x\in \Q/\Z$, we have
\begin{eqnarray*}
m_{A_L}^{\tor}(x)&=&\sum_{y\in \Q/\Z,\,e\cdot y=x}m_A^{\tor}(y)
\\ m_{A_L}^{\ab}(x)&=&\sum_{y\in \Q/\Z,\,e\cdot y=x}m_A^{\ab}(y)
\\ \dualm_{A_L}^{\ab}(x)&=&\sum_{y\in \Q/\Z,\,e\cdot y=x}\dualm_A^{\ab}(y)
\end{eqnarray*}
\end{prop}
\begin{proof}
We adopt the notations of Section \ref{subsec-neron}. Since the
multiplicity functions do not depend on the choice of the field
$K'$ where $A$ acquires semi-abelian reduction, we may assume that
$L$ is contained in $K'$. If $\zeta$ is a generator of
$\mu=G(K'/K)$, then the Galois group $G(K'/L)$ is generated by
$\zeta^e$. Now the result easily follows from the definition of
the multiplicity functions.
\end{proof}

\begin{prop}\label{prop-iso}
If $f:A_1\rightarrow A_2$ is an isogeny of tamely ramified abelian
$K$-varieties and if the degree $\deg(f)$ of $f$ is prime to $p$,
then
$$m_{A_1}^{\ab}=m_{A_2}^{\ab}\quad \mbox{ and }\quad  \dualm_{A_1}^{\ab}=\dualm_{A_2}^{\ab}.$$
\end{prop}
\begin{proof}
 We put $n=\deg(f)$.
 Since
$n$ is invertible in $K$, the morphism $f$ is separable, so that
$\ker(f)$ is \'etale over $k$. Thus $\ker(f)$ is a finite \'etale
$K$-group scheme of rank $n$.
 Every finite group scheme over a field is
killed by its rank.
 (See \cite[VII$_A$.8.5]{sga3.1}; in our case, the result is
elementary, because $\ker(f)$ is \'etale, so that we can reduce to
the case of a constant group by base change to an algebraic
closure of $K$. As an aside, we recall that Deligne has shown that
every {\em commutative} finite group scheme over an arbitrary base
scheme is killed by its rank \cite[p.4]{oort-tate}.) It follows
that $\ker(f)$ is contained in $_{n}(A_1)$. Hence, there exists an
isogeny $g:A_2\rightarrow A_1$ such that $g\circ f=n_{A_1}$.

Let $K'$ be a tame finite extension of $K$ such that  $A_1$ and
$A_2$ acquire semi-abelian reduction over $K'$, and denote by $R'$
the integral closure of $R$ in $K'$. We denote the N\'eron model
of $(A_i)\times_K K'$ by $\mathcal{A}'_i$, for $i=1,2$. The
morphisms $f\times_K K'$ and $g\times_K K'$ extend uniquely to
morphisms of $R'$-group schemes
$$\begin{array}{c}f':\mathcal{A}'_1\rightarrow \mathcal{A}_2'
\\g':\mathcal{A}'_2\rightarrow \mathcal{A}'_1.
\end{array}$$
For $i=1,2$, we denote by
 $B_i$ the abelian part of the semi-abelian $k$-variety $(\mathcal{A}'_i)^o_s$.
By functoriality of the Chevalley decomposition, $f'_s$ induces
 a morphism of $k$-group schemes 
$f'_B:B_1\rightarrow B_2$. Likewise, $g'_s$ induces a morphism of
$k$-group schemes 
$g'_B:B_2\rightarrow B_1$. Since $g'\circ f'$ is multiplication by
$n$, the same holds for 
 $g'_B\circ f'_B$. In
particular, $f'_B$ is an isogeny of degree prime to $p$.  Thus
$f'_B$ is a $\mu$-equivariant separable isogeny, so that
$\Lie(f'_B):\Lie(B_1)\rightarrow \Lie(B_2)$ is a $\mu$-equivariant
isomorphism, and $m_{A_1}^{\ab}=m_{A_2}^{\ab}$.

By \cite[p.\,143]{mumford-AV}, the dual morphism $(f'_B)^{\vee}$
is again an isogeny, and its kernel is the Cartier dual of the
kernel of $f'_B$. In particular, $f'_B$ and $(f'_B)^{\vee}$ have
the same degree, so that $(f'_B)^{\vee}$ is separable. Since it is
also equivariant for the left $\mu$-action on $B^{\vee}$, we find
that
 $\dualm_{A_1}^{\ab}=\dualm_{A_2}^{\ab}$.
\end{proof}
\begin{remark}
The same proof shows that $m_{A}^{\tor}$ is invariant under
isogenies of degree prime to $p$. We'll see in Corollary
\ref{cor-isogeny} that, more generally, the functions $m_A^{\tor}$
and $m_A^{\ab}+\dualm_A^{\ab}$  are invariant under {\em all}
 isogenies.
\end{remark}

\begin{cor}\label{cor-reflex} Let $A$ be a tamely ramified abelian $K$-variety.
If $k$ has characteristic zero, or $A$ is principally polarized,
then
$$m_A^{\ab}=m_{A^{\vee}}^{\ab}$$ and $$\dualm_A^{\ab}=(m_A^{\ab})^{\refl}.$$
\end{cor}
\begin{proof}
The first equality follows from Proposition \ref{prop-iso}.
Together with Proposition \ref{prop-dualeq}, it implies the second
equality.
\end{proof}

We will see in Theorem \ref{thm-weight} that, when $R$ is the ring
of germs of holomorphic functions at the origin of $\C$, the
equality
 $$\dualm_A^{\ab}=(m_A^{\ab})^{\refl}$$
expresses that the monodromy eigenvalues on the $(-1,0)$-component
of a certain limit mixed Hodge structure associated to $A$ are the
complex conjugates of the monodromy eigenvalues on the
$(0,-1)$-component. Corollary \ref{cor-reflex} generalizes this
Hodge symmetry.

\begin{question}\label{qu-reflex}
Is it true that
$$\dualm_A^{\ab}=(m_A^{\ab})^{\refl}$$
for every tamely ramified abelian $K$-variety $A$?
\end{question}

\section{Jumps and monodromy}\label{sec-jumpsmon}

\begin{prop}\label{prop-charpol}
Let $B$ be an abelian $k$-variety, and let $T$ be an algebraic
$k$-torus. Fix an element $e\in \N'$, and assume that $\mu_e(k)$
acts on $B$, resp. $T$ from the right. We consider the dual left
action of $\mu_e(k)$ on $B^{\vee}$. The functions
$m_1:=m_{\Lie(T),\mu_e(k)}$ and
$$m_2:=m_{\Lie(B),\mu_e(k)}+m_{\mu_e(k),\Lie(B^{\vee})}$$
are complete.

Moreover, for every prime $\ell$ invertible in $k$ and for each
generator $\zeta$ of $\mu_e(k)$, the characteristic polynomial
$P_1(t)$ of $\zeta$ on the $\ell$-adic Tate module
$$\cV_\ell T\cong\cT_\ell
T\otimes_{\Z_\ell} \Q_\ell$$  is equal to $Q_{m_1}(t)$ (in the
notation of Lemma \ref{lemm-comp}). Likewise, the characteristic
polynomial $P_2(t)$ of $\zeta$ on $\cV_\ell B$ is equal to
$Q_{m_2}(t)$.
\end{prop}
\begin{proof}
We denote by $$\rho:\Z[t]\to k[t]$$ the unique ring morphism that
maps $t$ to $t$. It is well-known that the characteristic
polynomials $P_1(t)$ and $P_2(t)$ belong to $\Z[t]$. For $P_1(t)$,
this follows from the canonical isomorphism
\begin{equation}\label{eq-tatetorus}\cV_\ell T \cong
\mathrm{Hom}_{\Z}(X(T),\Q_\ell(1)),\end{equation} where $X(T)$
denotes the character module of $T$.  For $P_2(t)$, it follows
from
 \cite[\S\,19,\,Thm.4]{mumford-AV}.

 Since $e$ is
invertible in $k$, $P_1(t)$ and $P_2(t)$ are products of $k$-tame
cyclotomic polynomials.  Thus, by Lemma \ref{lemm-comp} (and using
the notation introduced there), we only have to show the following
claims.

\vspace{5pt}

{\em Claim 1: The image of $P_1(t)$ under $\rho$ equals
$P_{m_1,\zeta}(t)$.}  Note that, by definition of the function
$m_1$, the polynomial $P_{m_1,\zeta}(t)$ is the characteristic
polynomial of the automorphism induced by $\zeta$ on $\Lie(T)$.
 Thus Claim 1 is an immediate consequence of \eqref{eq-tatetorus}
and the canonical isomorphism
$$\Lie(T)\cong \mathrm{Hom}_{\Z}(X(T),k).$$

\vspace{5pt}

 {\em Claim 2: The image of $P_2(t)$ under $\rho$ equals
$P_{m_2,\zeta}(t)$.} By definition of the function $m_2$, the
polynomial $P_{m_2, \zeta}(t)$ is
 the product of the characteristic polynomials of the automorphism induced by
$\zeta$ on $\Lie(B)$ and the dual automorphism on
$\Lie(B^{\vee})$. By \cite[5.1]{oda}, the Hodge-de Rham spectral
sequence of $B$ degenerates at $E_1$. This yields a natural short
exact sequence
\begin{equation*}\label{eq-hodgedr} 0\rightarrow
H^0(B,\Omega^1_B)\rightarrow H^1_{dR}(B)\rightarrow
H^1(B,\mathcal{O}_B)\rightarrow 0\end{equation*} where
$H^1_{dR}(B)$ is the degree one de Rham cohomology of $B$. We have
natural isomorphisms
\begin{eqnarray*}\label{eq-lieab}
 H^0(B,\Omega^1_B)&\cong&
 \Lie(B)^{\vee},
\\H^1(B,\mathcal{O}_B)&\cong& \Lie(B^{\vee}).
\end{eqnarray*}
(The first isomorphism follows from \cite[4.2.2]{neron}; the
second isomorphism from \cite[\S\,13, Cor.3]{mumford-AV}). Thus it
suffices to show that the image of $P_{2}(t)$ under $\rho$ is
equal to the characteristic polynomial of $\zeta$ on
$H^1_{dR}(B)$. As explained in the proof of \cite[5.12]{HaNi},
this can be deduced from the fact that \'etale cohomology is a
Weil cohomology, as well as de Rham cohomology (if $k$ has
characteristic zero) and crystalline cohomology (if $k$ has
characteristic $p>0$), so that the characteristic polynomials of
$\zeta$ on the respective cohomology spaces must coincide.
\end{proof}

For every $n\in \Z_{>0}$ and every $a\in \C$, we denote by
$\diag_n(a)$ the rank $n$ diagonal matrix with diagonal
$(a,\ldots,a)$, and by $\Jord_n(a)$ the rank $n$ Jordan matrix
with diagonal $(a,\ldots,a)$ and subdiagonal $(1,\ldots,1)$. For
any two complex square matrices $M$ and $N$, of rank $m$, resp.
$n$, we denote by $M\oplus N$ the rank $m+n$ square matrix
$$M\oplus N=\left(\begin{array}{cc}M & 0\\0& N\end{array}\right).$$
For every integer $q>0$, we put
$$\oplus^q M=\underbrace{M\oplus\cdots\oplus
M}_{q\mbox{\footnotesize{ times }}}.$$

\begin{definition}
For $i=1,2$, let
$$m_i:\Q/\Z\rightarrow \N$$ be a function with finite support. The Jordan matrix
$\Jord(m_1,m_2)$ associated to the couple $(m_1,m_2)$ is the
complex square matrix of rank $\|m_1\|+2\cdot \|m_2\|$ given by
\begin{eqnarray*}
\Jord(m_1,m_2)&=&\bigoplus_{x\in \Supp(m_1)}\left(
\diag_{m_1(x)}(\exp(2\pi i x))\right)
\\ & \oplus &\bigoplus_{y\in
\Supp(m_2)}\left(\oplus^{m_2(y)} \Jord_2(\exp(2\pi i
y))\right)\end{eqnarray*} where we ordered the set $\Q/\Z$ using
the bijection $\Q\cap [0,1[\to \Q/\Z$ and the usual ordering on
$[0,1[$.
\end{definition}

\begin{lemma}\label{lemm-linalg}
Let $V$ be a finite dimensional vector space over an algebraically
closed field $F$ of characteristic zero, and let $M$ be an
endomorphism of $V$. Let $d>0$ be an integer such that $M^d$ is
unipotent. Set $$W=\{v\in V\,|\,M^d(v)=v\}$$ and assume that $M^d$
acts trivially on $V/W$ and that there exists an $M$-equivariant
isomorphism between $(V/W)^{\vee}$ and an $M$-stable subspace $W'$
of $W$.

Then the endomorphism $M$ of $V$ has the following Jordan form:
for every eigenvalue $\alpha$ of $M$ on $W'$ (counted with
multiplicities), there is a Jordan block of size two with
eigenvalue $\alpha$, and for every eigenvalue $\beta$ of $M$ on
$W/W'$ (counted with multiplicities), there is a Jordan block of
size one with eigenvalue $\beta$.
\end{lemma}
\begin{proof}
Since $M^d$ acts trivially on $W$ and $V/W$, we know that
$(M^d-\Id)^2=0$ on $V$, so that the minimal polynomial of $M$
divides $(t^d-1)^2$ and the Jordan blocks of $M$ have size at most
two.
 The subspace $W$ of $V$ is generated by all the eigenvectors of
$M$. Thus the dimension of $V/W$ is equal to the number of Jordan
blocks of $M$ of size two, and the eigenvalues of these Jordan
blocks are precisely the eigenvalues of $M$ on $V/W$, or,
equivalently, $W'$. It follows that the eigenvalues of $M$ on $V$
corresponding to a Jordan block of size one are the eigenvalues of
$M$ on $W/W'$.
\end{proof}

\begin{theorem}\label{thm-jord}
We fix an embedding $\Q_\ell \hookrightarrow \C$. If $A$ is a
tamely ramified abelian $K$-variety, then the monodromy action of
$\sigma$ on $H^1(A\times_K K^t,\Q_\ell)$ has Jordan form
$$\Jord(m_A^{\ab}+\dualm_A^{\ab},m_A^{\tor}).$$
Moreover, the functions $m_A^{\tor}$ and
$m_A^{\ab}+\dualm_A^{\ab}$ are complete.
\end{theorem}
\begin{proof}
We adopt the notations of Section \ref{subsec-neron}. We denote by
$\cT_{\ell}A$ the $\ell$-adic Tate module of $A$. We put
$I=G(K^s/K)$ and $I'=G(K^s/K')$. Recall that there exists a
canonical $I$-equivariant isomorphism
\begin{equation}\label{eq-tate}
H^1(A\times_K K^s,\Q_\ell)\cong Hom_{\Z_\ell}(\cT_\ell A,\Q_\ell)
\end{equation}
 (see \cite[15.1]{milne-abelian}). Since $A$ is tamely ramified, the wild inertia subgroup $P\subset
I$ acts trivially on $H^1(A\times_K K^s,\Q_\ell)$ and $\cT_\ell
A$, so that the $I$-action on these modules factors through an
action of $I/P=G(K^t/K)$.

Since $P$ is a $p$-group and $p$ is prime to $\ell$, there exists
for every $K$-variety $X$ and every integer $i\geq 0$ a canonical
$G(K^t/K)$-equivariant isomorphism
$$H^i(X\times_K K^t,\Q_\ell)\cong H^i(X\times_K K^s,\Q_\ell)^P$$
(see \cite[I.2.7.1]{sga7a}). In our case, this yields a canonical
$G(K^t/K)$-equivariant isomorphism
\begin{equation}\label{eq-tame}
H^1(A\times_K K^s,\Q_\ell)= H^1(A\times_K K^s,\Q_\ell)^P\cong
H^1(A\times_K K^t,\Q_\ell).\end{equation} By \eqref{eq-tate} and
\eqref{eq-tame}, it suffices to show that the action of $\sigma$
on $$\cV_\ell A=\cT_\ell A\otimes_{\Z_\ell}\Q_\ell$$ has Jordan
form
$$\Jord(m_A^{\ab}+\dualm_A^{\ab},m_A^{\tor})$$
and that $m_A^{\ab}+\dualm_A^{\ab}$ and $m_A^{\tor}$ are complete.

Consider the filtration \begin{equation}\label{eq-filt} (\cT_\ell
A)^{\et} \subset (\cT_\ell A)^{\ef} \subset \cT_\ell
A\end{equation} from \cite[IX.4.1.1]{sga7a}, with $(\cT_\ell
A)^{\ef}$ the {\em essentially fixed part} of the Tate module
$\cT_\ell A$, and $(\cT_\ell A)^{\et}$ the {\em essentially toric
part}.  By definition, $$(\cT_\ell A)^{\ef}=(\cT_\ell A)^{I'}$$
and $(\cT_\ell A)^{\et}$ is stable under the action of $I$ on
$\cT_\ell A$. We denote by
\begin{equation}\label{eq-filt2} (\cV_\ell
A)^{\et} \subset (\cV_\ell A)^{\ef}=(\cV_\ell A)^{I'} \subset
\cV_\ell A\end{equation} the filtration obtained from
\eqref{eq-filt} by
 tensoring with $\Q_\ell$.
 By \cite[IX.4.1.2]{sga7a} there exists an
$I$-equivariant isomorphism
\begin{equation}\label{eq-dual}
\cV_\ell A/(\cV_\ell A)^{\ef}\cong ((\cV_\ell A)^{\et})^{\vee}.
\end{equation}
In particular, $I'$ acts trivially on $\cV_\ell A/(\cV_\ell
A)^{\ef}$. It follows that the $I'$-action on $\cV_\ell A$  is
unipotent of level $\leq 2$ (this means that for every element
$\gamma$ of $I'$, the endomorphism $(\gamma-\Id)^2$ of $\cV_\ell
A$ is zero), and that the $I$-action on $(\cV_\ell A)^{\et}$ and
$\cV_\ell A/(\cV_\ell A)^{\ef}$ factors through an action of
$I/I'\cong \mu=\mu_d(k)$, where $d=[K':K]$. We denote by
$\overline{\sigma}$ the image of $\sigma$ under the projection
$G(K^t/K)\to \mu$.

 The element $\sigma^d$ is a topological generator of $I'/P$.
Combining \eqref{eq-filt2} and \eqref{eq-dual} and applying Lemma
\ref{lemm-linalg} to the $\sigma$-action on $\cV_\ell A$, we see
that
%
 it suffices to prove the
following claims: \begin{enumerate} \item the
$\overline{\sigma}$-action on $(\cV_\ell A)^{\et}$ has Jordan form
$\Jord(m_A^{\tor},0)$, and $m_A^{\tor}$ is complete, \item the
$\overline{\sigma}$-action on $(\cV_\ell A)^{\ef}/(\cV_\ell
A)^{\et}$ has Jordan form $\Jord(m_A^{\ab}+\dualm^{\ab}_A,0)$, and
$m_A^{\ab}+\dualm^{\ab}_A$ is complete.
\end{enumerate}

Since $\overline{\sigma}^d$ is the identity, the Jordan forms of
the $\overline{\sigma}$-actions in (1) and (2) are diagonal
matrices. By \cite[IX.4.2.7 and IX.4.2.9]{sga7a} there exist
$\mu$-equivariant isomorphisms
\begin{eqnarray*}
(\cV_\ell A)^{\et}&\cong& \cV_\ell T
\\ (\cV_\ell A)^{\ef}/(\cV_\ell A)^{\et}&\cong& \cV_\ell B
\end{eqnarray*}
so that claims (1) and (2) follow from Proposition
\ref{prop-charpol}.
\end{proof}


\begin{cor}\label{cor-isogeny}
The functions $m_A^{\ab}+\dualm_A^{\ab}$ and $m_A^{\tor}$ are
invariant under isogeny. In particular,
$m_A^{\tor}=m_{A^{\vee}}^{\tor}$, and
$$m_A^{\ab}+\dualm_A^{\ab}=m_{A^{\vee}}^{\ab}+\dualm_{A^{\vee}}^{\ab}.$$
The multiplicity function $m_A$, the jumps of $A$ (with their
multiplicities), the elementary divisors of $A$ and the base
change conductor $c(A)$ are invariant under isogenies of degree
prime to $p$.
\end{cor}
\begin{proof}
This follows from Propositions \ref{prop-jumps} and
\ref{prop-iso}, and Theorem \ref{thm-jord}.
\end{proof}
\begin{remark}\label{rem-isog} Beware that the base change conductor, and thus the function
$m_A^{\ab}$, of a tamely ramified abelian $K$-variety $A$ can
change under isogenies of degree $p$, if $p>0$; see
\cite[6.10.2]{chai} for an example.
\end{remark}

\begin{cor}\label{cor-mult}
Using the notations of Proposition \ref{prop-jumps}, the base
change conductor of a tamely ramified abelian $K$-variety $A$ is
given by
$$c(A)=\frac{1}{2}(t_{\pot}(A)-t(A))+\sum_{x\in ]0,1[\cap
\Q}m_A^{\ab}(x)x.$$ In particular, if $A$ has potential purely
multiplicative reduction (which means that $a_{\pot}(A)=0$), then
$$c(A)=\frac{u(A)}{2}.$$
\end{cor}
\begin{proof} By Proposition \ref{prop-jumps}, we know that
$$c(A)=\sum_{x\in ]0,1[\cap
\Q}m_A^{\tor}(x)x+\sum_{x\in ]0,1[\cap \Q}m_A^{\ab}(x)x.$$ Since
$m^{\tor}_A$ is complete, we have that
\begin{eqnarray*}\sum_{x\in ]0,1[\cap
\Q}m_A^{\tor}(x)x&=&\frac{1}{2}\left(\sum_{x\in ]0,1[\cap
\Q}m_A^{\tor}(x)x+\sum_{x\in ]0,1[\cap
\Q}m_A^{\tor}(x)(1-x)\right)
\\[1.5ex] &=& \frac{1}{2}\left(\sum_{x\in ]0,1[\cap \Q}m_A^{\tor}(x)\right)
\\[1.5ex] &=& \frac{1}{2}(\,\|m_A^{\tor}\|-m_A^{\tor}(0))
\\[1.5ex] &=& \frac{1}{2}(t_{\pot}(A)-t(A))
\end{eqnarray*} where the last equality follows from Proposition
\ref{prop-basic}. If $a_{\pot}(A)=0$, then it follows from
Proposition \ref{prop-basic} that $m_A^{\ab}=0$ and $a(A)=0$, so
that
$$c(A)=\frac{1}{2}(t_{\pot}(A)-t(A))=\frac{1}{2}(\dim(A)-t(A))=\frac{u(A)}{2}.$$
\end{proof}

\begin{remark}
If $A$ has potential purely multiplicative reduction, then
Corollary \ref{cor-mult} can be viewed as a special case of Chai's
result that for {\em every} abelian $K$-variety $A$ (not
necessarily tamely ramified) with potential purely multiplicative
reduction, the base change conductor $c(A)$ equals one fourth of
the Artin conductor of $\cV_\ell A$ \cite[5.2]{chai}. If $A$ is
tamely ramified, then this Artin conductor is simply the dimension
of $\cV_\ell A/((\cV_\ell A)^{ss})^I$, where $I=G(K^s/K)$ and
$(\cV_\ell A)^{ss}$ is the semi-simplification of the $\ell$-adic
$I$-representation $\cV_\ell A$. This value is precisely the
number of eigenvalues of $\sigma$ (counted with multiplicities) 
that are different from one. Combining Proposition
\ref{prop-basic} with Theorem \ref{thm-jord}, we find that the
Artin conductor of $\cV_\ell A$ equals
$$2\dim(A)-m^{\ab}_A(0)-\dualm^{\ab}_A(0)-2m_A^{\tor}(0)= 2u(A).$$
\end{remark}

\begin{cor}\label{cor-charpol}
Let $A$ be a tamely ramified abelian $K$-variety, and let $e$ be
the degree of the minimal extension of $K$ where $A$ acquires
semi-abelian reduction. Fix a primitive $e$-th root of unity $\xi$
in an algebraic closure $\Q^a$ of $\Q$. The characteristic
polynomial
$$P_{\sigma}(t)=\det(t\cdot \Id - \sigma\,|\,H^1(A\times_K
K^t,\Q_\ell))$$ of $\sigma$ on $H^1(A\times_K K^t,\Q_\ell)$ is
given by
$$P_{\sigma}(t)=\prod_{i\in ((1/e)\Z)/\Z}(t-\xi^{e\cdot
i})^{m^{\ab}_A(i)+\dualm_A^{\ab}(i)+2m_A^{\tor}(i)} \quad \in
\Z[t].$$
\end{cor}
\begin{proof}
This is an immediate consequence of Theorem \ref{thm-jord}.
\end{proof}

\begin{cor}
Let $A$ be a tamely ramified abelian $K$-variety. Assume either
that $k$ has characteristic zero, or that $A$ is principally
polarized. Then  $m_A^{\ab}$ and $\dualm_A^{\ab}$ are
semi-complete, and the monodromy action of $\sigma$ on
$H^1(A\times_K K^t,\Q_\ell)$ has Jordan form
$$\Jord(m_A^{\ab}+(m_A^{\ab})^{\refl},m_A^{\tor}).$$
In the notation of Corollary \ref{cor-charpol}, we have
$$P_{\sigma}(t)=\prod_{i\in ((1/e)\Z)/\Z}(t-\xi^{e\cdot
i})^{m^{\ab}_A(i)+m_A^{\ab}(-i)+2m_A^{\tor}(i)} \quad \in \Z[t].$$
\end{cor}
\begin{proof} Semi-completeness of $m_A^{\ab}$ and
$\dualm_A^{\ab}$ follows from Corollary \ref{cor-reflex} and
Theorem \ref{thm-jord}. The remainder of the statement is a
combination of Corollaries \ref{cor-reflex} and \ref{cor-charpol}.
\end{proof}

\section{Potential toric rank and Jordan blocks}
\subsection{The weight filtration associated to a nilpotent
operator}\label{subsec-weight} Throughout this section, we fix a
field $F$ of characteristic zero and a finite dimensional vector
space $V$ over $F$. For every endomorphism $M$ on $V$, we consider
its Jordan-Chevalley decomposition $$M=M_s+M_n$$ with $M_s$ the
semi-simple part of $M$ and $M_n$ its nilpotent part.

We recall the following well-known property.
\begin{prop}\label{prop-weight}
 Let $N$ be a nilpotent endomorphism of $V$.
Let $w$ be an integer. There exists a unique finite ascending
filtration $W=(W_i,\,i\in \Z)$ on $V$ such that
\begin{enumerate}
\item $NW_i\subset W_{i-2}$ for all $i$ in $\Z$, \item the
morphism of vector spaces $$\Gr^{W}_{w+\alpha}V\to
\Gr^{W}_{w-\alpha}V$$ induced by $N^{\alpha}$ is an isomorphism
for every integer $\alpha>0$.
\end{enumerate}
The filtration $W$ is called the {\em weight filtration} centered
at $w$ associated to the nilpotent operator $N$.
\end{prop}
\begin{proof} See, for instance,
\cite[1.6.1]{deligne-weilII}.\end{proof}

It is clear from the definition that the weight filtration
centered at another integer $w'$ is the shifted filtration
$W'_{\bullet}=W_{\bullet- w'+w}$. We define the {\em amplitude} of
the filtration $W$ in Proposition \ref{prop-weight} as the
smallest integer $n\geq 0$ such that $W_{w+n}=V$. This value does
not depend on the choice of the central weight $w$. The amplitude
is related to sizes of Jordan blocks in the following way.

\begin{prop}\label{prop-weightjord}
Let $M$ be an endomorphism of $V$. We denote by $a$ the amplitude
of the weight filtration $W$ associated to $M_n$ (centered at any
weight $w\in \Z$). Then $a+1$ is the maximum of the ranks of the
Jordan blocks of $M$.
\end{prop}
\begin{proof}
We may assume that $w=0$ and that $M=M_n$. Then the proposition is
an immediate consequence of the explicit description of the weight
filtration in \cite[1.6.7]{deligne-weilII}.
\end{proof}

\begin{prop}\label{prop-weight2}
Let $N$ be a nilpotent endomorphism of $V$, and denote by $W$ the
associated weight filtration centered at $w\in \Z$.

\begin{enumerate}
\item If $N'$ is another nilpotent operator on $V$ such that
$$(N-N')W_i\subset W_{i-3}$$ for all $i\in \Z$, then the weight
filtration associated to $N'$ centered at $w$ coincides with $W$.
\item The weight filtration $W$ does not change if we multiply $N$
with an automorphism $S$ of $V$ that commutes with $N$.
\end{enumerate}
\end{prop}
\begin{proof}
(1) This follows immediately from the definition of the weight
filtration in Proposition \ref{prop-weight}.

(2) Clearly, the filtration $(S(W_i),\,i\in \Z)$ on $V$ also
satisfies properties (1) and (2) in Proposition \ref{prop-weight},
so that $S(W_i)=W_i$ for all $i$ in $\Z$ by uniqueness of the
weight filtration. This implies at once that $W$ coincides with
the weight filtration associated to $NS$ centered at $w$.
\end{proof}

\begin{definition}
Let $W=(W_i,\,i\in \Z)$ be an ascending filtration on $V$. The
dual filtration $W^{\vee}$ on $V^{\vee}$ is the ascending
filtration defined by
$$(W^{\vee})_i=(W_{-i-1})^{\bot}$$ for all $i$ in $\Z$.

For every integer $j\geq 0$, the degree $j$ exterior power of $W$
is the ascending filtration $\wedge ^jW$ on $\wedge^j V$ given by
$$(\wedge^j W)_i=\sum_{i_1+\ldots+i_j=i}W_{i_1}\wedge \cdots \wedge
W_{i_j}.$$

If $V'$ is another finite dimensional vector space over $F$,
endowed with an ascending filtration $W'$, then the tensor product
of $W$ and $W'$ is the ascending filtration $W\otimes W'$ on
$V\otimes V'$ given by
$$(W\otimes W')_i=\sum_{i_1+i_2=i}W_{i_1}\otimes W'_{i_2}.$$
\end{definition}

\begin{prop}\label{prop-weightlin} Let $V$ and $V'$ be finite
dimensional vector spaces over $F$, endowed with nilpotent
operators $N$ and $N'$, respectively. We denote by $W$ and $W'$
the associated weight filtrations on $V$ and $V'$, centered at
integers $w$ and $w'$.
\begin{enumerate}
\item The weight filtration on $V^{\vee}$ centered at $-w$
associated to the nilpotent operator $-N^{\vee}$ is the dual
$W^{\vee}$ of the weight filtration $W$.

\item The weight filtration on $V\otimes V'$ centered at $w+w'$
associated to the nilpotent operator $N\otimes \Id+\Id \otimes N'$
is the tensor product $W\otimes W'$ of the weight filtrations $W$
and $W'$.

\item For every integer $j>0$, the weight filtration on $\wedge^j
V$ centered at $w\cdot j$ associated to the nilpotent operator
$$N^{(\wedge j)}=N\wedge \Id\wedge \cdots\wedge \Id+\Id \wedge N\wedge \Id\wedge \cdots \wedge \Id+\ldots +\Id \wedge \cdots \wedge \Id \wedge N $$
is the exterior power $\wedge^j W$ of the weight filtration $W$.
\end{enumerate}
\end{prop}
\begin{proof}
Point (1) and (2) are proven in \cite[1.6.9]{deligne-weilII},
using the theorem of Jacobson-Morosov. Point (3) can be proven in
 exactly the same way, since the
 morphism of linear groups
 $$h:\mathrm{GL}(V)\to \mathrm{GL}(\wedge ^j V):M\mapsto \wedge^j
 M$$  induces a morphism of Lie algebras $$\mathrm{Lie}(h):\mathrm{End}(V)\to \mathrm{End}(\wedge^j V)$$ that sends $N$ to $N^{(\wedge j)}$.
 \end{proof}
\begin{cor}\label{cor-extprod}
Let $M$ be an automorphism of $V$, and consider an integer $w$ and
an integer $j>0$. We denote by $W$ the weight filtration centered
at $w$ associated to the nilpotent operator $M_n$ on $V$.

\begin{enumerate}
\item Let $V'$ be another finite dimensional vector space, endowed
with an automorphism $M'$.  If $W'$ is the weight filtration
associated to $M'_n$ centered at $w'\in \Z$, then $W\otimes W'$ is
the weight filtration centered at $w+w'$ associated to the
nilpotent operator $(M\otimes M')_n$ on $V\otimes V'$.

\item The exterior power filtration $\wedge^jW$ is the weight
filtration centered at $w\cdot j$ associated to the nilpotent
operator $(\wedge^j M)_n$ on $\wedge^jV$.
\end{enumerate}
\end{cor}
\begin{proof}
(1) By Proposition \ref{prop-weight2}(2), we may assume that $M$
and $M'$ are unipotent, since multiplying these operators with
$M_s^{-1}$ and $(M'_s)^{-1}$, respectively, has no influence on
the weight filtrations that we want to compare. Then $M=\Id+M_n$
and $M'=\Id+M'_n$, so that
 $$(M\otimes
M')_n-\Id\otimes M'_n-M_n\otimes \Id=M_n\otimes M'_n.$$ It follows
that $$((M\otimes M')_n-\Id\otimes M'_n-M_n\otimes \Id)(W\otimes
W')_{i}\subset (W\otimes W')_{i-4}$$ for all $i\in \Z$. The result
now follows from
 Propositions \ref{prop-weight2}(1) and \ref{prop-weightlin}(2).

(2) The proof is completely similar to the proof of (1): one
reduces to the case where $M$ is unipotent, and one shows by
direct computation that
$$(\wedge^j M)_n- (M_n)^{(\wedge j)}$$ shifts weights by at least
$-4$.
\end{proof}

The following lemma and proposition will allow us to compute the
maximal size of certain Jordan blocks of monodromy on the
cohomology of a tamely ramified abelian $K$-variety (Theorem
\ref{thm-blocksize}).

\begin{lemma}\label{lemm-onejord}
Let $F$ be an algebraically closed field of characteristic zero,
and let $V\neq \{0\}$ be a finite dimensional vector space over
$F$. Let $M$ be an automorphism of $V$, with Jordan form
$$\Jord_{m}(\xi)$$
where $m\in \Z_{>0}$ and $\xi\in F^{\times}$. Then for every
integer $j$ in $[1,m]$ and every integer $w$, the weight
filtration centered at $w$ associated to the nilpotent operator
$(\wedge^j M)_n$ on $\wedge^j V$ has amplitude $m(m-j)$.
\end{lemma}
\begin{proof}
We may assume that $w=0$. Denote by $W$ the weight filtration
associated to $M_n$ centered at $0$. By Corollary
\ref{cor-extprod}, the weight filtration associated to $(\wedge^j
M)_n$ centered at $0$ coincides with the exterior power filtration
$\wedge^j W$.

By the explicit description of the weight filtration in
\cite[1.6.7]{deligne-weilII}, the dimension of
$\mathrm{Gr}^{W}_\alpha V$ is one if $\alpha$ is an integer in
$[1-m,m-1]$ such that $\alpha-m$ is odd, and zero in all other
cases. This easily implies that $\wedge^j W$ has amplitude
$$(m-1)+(m-3)+\ldots + (m-2j+1)=m(m-j).$$
\end{proof}

\begin{prop}\label{prop-wedge}
Let $F$ be an algebraically closed field of characteristic zero,
and let $V\neq \{0\}$ be a finite dimensional vector space over
$F$. Let $M$ be an automorphism of $V$, with Jordan form
$$\Jord_{m_1}(\xi_1)\oplus\cdots \oplus \Jord_{m_q}(\xi_q)$$
where $q\in \Z_{>0}$, $m\in (\Z_{>0})^q$ and $\xi_i\in F^{\times}$
for $i=1,\ldots,q$.

We fix an integer $j>0$. For every element $\zeta$ of $F$, we
denote by $\mathrm{max}_{\zeta}$ the maximum of the ranks of the
Jordan blocks of
 $\wedge^j M$ on $\wedge^j V$ with
eigenvalue $\zeta$. If we denote by $\mathscr{S}$ the set of
tuples $s\in \N^q$ such that $\|s\|=j$ and $s_i\leq m_i$ for each
$i\in \{1,\ldots,q\}$, then
$$\mathrm{max}_{\zeta}=\max \{1+\sum_{i=1}^{q}s_i(m_i-s_i)\,|\,s\in \mathscr{S},\ \prod_{i=1}^q (\xi_i)^{s_i}=\zeta\}$$
for every $\zeta\in F$, with the convention that $\max
\emptyset=0$.
%
\end{prop}
\begin{proof}
%
 We can write
$$V=V_1\oplus \cdots\oplus V_q$$ such that $M(V_i)\subset V_i$ for each $i$ and such that the restriction
$M_i$ of $M$ to $V_i$ has Jordan form $\Jord_{m_i}(\xi_i)$. If we
put
$$V_s=(\wedge^{s_1}V_1)\otimes\cdots \otimes(\wedge^{s_q}V_q)$$ for each $s\in \mathscr{S}$,
 then we have a canonical
isomorphism
$$\wedge^j V\cong \bigoplus_{s\in \mathscr{S}}V_s$$
such that every summand $V_s$ is stable under $\wedge^j M$ and the
restriction of $\wedge^jM$ to $V_s$ equals
$$(\wedge^{s_1}M_1)\otimes\cdots \otimes (\wedge^{s_q}M_q).$$ The
 automorphism $\wedge^j M$ has a unique eigenvalue on $V_s$, which
is equal to
$$\prod_{i=1}^q (\xi_i)^{s_i}.$$
By Proposition \ref{prop-weightjord}, Corollary
\ref{cor-extprod}(1) and Lemma \ref{lemm-onejord}, the maximal
 rank of a Jordan block of $\wedge^j M$ on $V_s$ equals
$$1+\sum_{i=1}^{q}s_i(m_i-s_i).$$ This yields the desired
formula for $\mathrm{max}_{\zeta}$.
\end{proof}

\subsection{The strong form of the monodromy conjecture}
\begin{theorem}\label{thm-blocksize}
Let $A$ be a tamely ramified abelian $K$-variety of dimension $g$.
 For every embedding of
$\Q_\ell$ in $\C$, the value $\alpha=\exp (2\pi i c(A))$ is an
eigenvalue of $\sigma$ on $H^g(A\times_K K^t,\Q_\ell)$. Each
Jordan block of $\sigma$ on $H^g(A\times_K K^t,\Q_\ell)$ has rank
at most $t_{\pot}(A)+1$, and $\sigma$ has a Jordan block with
eigenvalue $\alpha$ on $H^g(A\times_K K^t,\Q_\ell)$ with rank
$t_{\pot}(A)+1$.
\end{theorem}
\begin{proof}
Since $A$ is tamely ramified, we have a canonical
$G(K^t/K)$-equivariant isomorphism of $\Q_\ell$-vector spaces
$$H^g(A\times_K K^t,\Q_\ell)\cong \bigwedge^g H^1(A\times_K
K^t,\Q_\ell).$$ By Theorem \ref{thm-jord}, the monodromy operator
$\sigma$ has precisely $\|m_A^{\tor}\|$ Jordan blocks of size $2$
on $H^1(A\times_K K^t,\Q_\ell)$, and no larger Jordan blocks. It
follows from Proposition \ref{prop-wedge} that the size of the
Jordan blocks of $\sigma$ on $H^g(A\times_K K^t,\Q_\ell)$ is
bounded by $1+\|m_A^{\tor}\|$. By Proposition \ref{prop-basic}, we
know that $\|m_A^{\tor}\|=t_{\pot}(A)$.

 By Proposition \ref{prop-jumps}, the image in
$\Q/\Z$ of the base change conductor $c(A)$ equals
$$\sum_{x\in \Q/\Z}((m_A^{\tor}(x)+m_A^{\ab}(x))\cdot x)$$ and by
Proposition \ref{prop-basic}, we have
$$\sum_{x\in \Q/\Z}(m_A^{\tor}(x)+m_A^{\ab}(x))=g.$$
 Hence, by  Theorem \ref{thm-jord} and Proposition \ref{prop-wedge}, $\sigma$ has
 a Jordan block of size $1+t_{\pot}(A)$ with eigenvalue $\alpha$  on $H^g(A\times_K K^t,\Q_\ell)$.
\end{proof}

\section{Limit Mixed Hodge structure}
Let $A$ be a tamely ramified abelian $K$-variety of dimension $g$.
Theorem \ref{thm-jord} shows that the couple of functions
$(m_A^{\tor},m_A^{\ab}+\dualm_A^{\ab})$ and the Jordan form of
$\sigma$ on the tame degree one cohomology of $A$ determine each
other. It does not tell us how to recover $m_A^{\ab}$ and
$\dualm_A^{\ab}$ individually from the cohomology of $A$.

 In this section, we assume
that $A$ is obtained by base change from a family of abelian
varieties over a smooth complex curve. We will give an
interpretation of the multiplicity functions $m_A^{\ab}$,
$\dualm^{\ab}_A$ and $m_A^{\tor}$ in terms of
 the limit mixed Hodge structure
of the family.

\subsection{Limit mixed Hodge structure of a family of abelian
varieties}\label{subsec-limMHS} Let $\overline{S}$ be a connected
smooth complex algebraic curve, let $s$ be a closed point on
$\overline{S}$, and choose a local parameter $t$ on $\overline{S}$
at $s$. We put $K=\C((t))$, $R=\C[[t]]$ and
$S=\overline{S}\setminus \{s\}$. Let
$$f:X\rightarrow S$$ be a smooth projective family of abelian varieties
over $S$, of relative dimension $g$, and put
$$A=X\times_{S}\Spec K.$$ We choose an extension of $f$ to a
 flat projective morphism
$$\overline{f}:\overline{X}\rightarrow \overline{S},$$
 and we denote by
$\overline{X}_s$ the fiber of $\overline{f}$ over the point $s$.

 We denote by $(\cdot)^{\an}$ the complex analytic GAGA functor, but we will usually omit
  it from the notation if no confusion can occur. For instance,
  when we speak of the sheaf $R^if_*(\Z)$, it should be clear that
  we mean $R^if^{\an}_*(\Z)$.

For every $i\in \N$, we consider the degree $i$ limit cohomology,
resp.
 homology,
\begin{eqnarray*}
H^i(X_{\infty},\Z)&:=&\mathbb{H}^{i}(\overline{X}_s(\C),R\psi_{\overline{f}}(\Z))
\\
H_i(X_{\infty},\Z)&:=&\mathbb{H}^{2g-i}(\overline{X}_s(\C),R\psi_{\overline{f}}(\Z))(g)=H^{2g-i}(X_{\infty},\Z)(g)
\end{eqnarray*}
of $\overline{f}$ at $s$. Here
$$R\psi_{\overline{f}}(\Z)\in D^b_c(\overline{X}_s(\C),\Z)$$ denotes
the complex of nearby cycles associated to $\overline{f}^{\an}$.
For every $i$ in $\N$, the $\Z$-module
 $H^i(X_{\infty},\Z)$ is non-canonically isomorphic to
 $H^i(X_z(\C),\Z)$, where $z$ is any point of $S(\C)$ and $X_z$ is the fiber of $f$ over $z$ \cite[XIV.1.3.3.2]{sga7b}.
 Likewise, by Poincar\'e duality, $H_i(X_{\infty},\Z)$ is non-canonically isomorphic to
 $H_i(X_z(\C),\Z)$. The limit cohomology and homology are independent of the chosen
compactification $\overline{f}$, as can be deduced from
\cite[4.2.11]{dimca} by dominating two compactifications by a
third one.

 We
set
\begin{eqnarray*}
H^i(X_\infty,\Q)&:=&H^i(X_\infty,\Z)\otimes_{\Z} \Q
\\H^i(X_\infty,\C)&:=&H^i(X_\infty,\Z)\otimes_{\Z}\C
\end{eqnarray*}
and we use similar notations for the limit homology
$H_i(X_{\infty},\ast)$.
 For all $i\in \N$, we denote by $M$ the monodromy operators on
$H_i(X_{\infty},\Z)$ and $H^i(X_{\infty},\Z)$, and by $N$ the
logarithm of the unipotent part $M_u$ of $M$.

The $\Z$-modules $H^i(X_\infty,\Z)$ and $H_i(X_\infty,\Z)$ carry
natural mixed Hodge structures, which are the limits at $s$ of the
variations of Hodge structures
$$R^if_*(\Z),\mbox{ resp. }  R^{2g-i}f_*(\Z)(g),$$
on $S$.
 The existence of these limit mixed Hodge structures was conjectured by Deligne, and they were constructed by Schmid
 \cite{schmid} and Steenbrink
\cite{steenbrink-limit}.
  The weight filtrations on $H^i(X_\infty,\Q)$ and
 $H_i(X_\infty,\Q)$ are the weight filtrations centered at $i$,
 resp. $-i$, associated to the nilpotent operator $N$.
    We will briefly recall Schmid's
construction of the limit Hodge filtration  below. The action of
the semi-simple part $M_{s}$ of $M$ on $H_i(X_{\infty},\Q)$ and
 $H^i(X_{\infty},\Q)$ is a morphism of rational mixed Hodge
structures, by \cite[2.13]{steenbrink-vanish}.



For every $i\in \N$, there exists an isomorphism of
$\Q_\ell$-vector spaces
\begin{equation}\label{comparcoh}
H^i(X_{\infty},\Z)\otimes_{\Z}\Q_\ell\cong H^i(A\times_K
K^a,\Q_\ell)\end{equation} such that the monodromy action on the
left hand side corresponds to the action of the canonical
topological generator of $G(K^a/K)\cong \widehat{\Z}(1)(\C)$ on
the right hand side. This follows from Deligne's comparison
theorem for $\ell$-adic versus complex analytic nearby cycles
\cite[XIV.2.8]{sga7b}. Thus, if $K'$ is a finite extension of $K$
such that $A\times_K K'$ has semi-abelian reduction and if we set
$d=[K':K]$, then $(M_s)^d$ is the identity on $H^i(X_{\infty},\Q)$
and $H_i(X_{\infty},\Q)$ for all $i\geq 0$. Identifying $M_s$ with
the canonical generator $\exp(2\pi i/d)$ of $\mu_d(\C)$, we obtain
an action of $\mu_d(\C)$ on the rational mixed Hodge structures
 $H^i(X_{\infty},\Q)$
and $H_i(X_{\infty},\Q)$, for all $i\geq 0$.


For the definition of the dual and the exterior powers of a mixed
Hodge structure, we refer to \cite[3.2]{peters-steenbrink}.
\begin{prop}\label{prop-cup}
\begin{enumerate}
\item For every $i\in \N$, there exists a natural isomorphism of
mixed Hodge structures
$$
\wedge^i_{\Z} H^1(X_\infty,\Z)\rightarrow H^i(X_\infty,\Z)$$ that
is compatible with the action of $M$ on the underlying
$\Z$-modules. \item For every $i\in \N$, there exists a natural
 isomorphism of mixed Hodge structures
$$ H^i(X_\infty,\Z)^{\vee}\to H_i(X_\infty,\Z)$$ that is
compatible with the action of $M$  on the underlying $\Z$-modules.
\end{enumerate}
\end{prop}
\begin{proof}
(1) The cup product defines a morphism
\begin{equation}\label{cup}
\wedge^i_{\Z}R^1f_*(\Z)\to R^if_*(\Z) \end{equation} of sheaves on
$S$. This is an isomorphism on every fiber, by the proper base
change theorem and \cite[p.\,3]{mumford-AV}, and thus an
isomorphism of sheaves. Moreover, it is an isomorphism of
variations of Hodge structures, because the cup product defines a
morphism of pure Hodge structures on the cohomology of every fiber
of $f$ \cite[5.45]{peters-steenbrink}.

 Looking at Schmid's
construction of the limit mixed Hodge structure in \cite{schmid},
one checks in a straightforward way that taking the limit of a
variation of Hodge structures commutes with taking exterior
powers. Compatibility of the Hodge filtrations is easy, since the
exterior power defines a holomorphic map between the relevant
 classifying spaces.
 The
compatibility of the weight filtrations follows from Corollary
\ref{cor-extprod}.

 Thus, taking the limit at $s$ of the isomorphism \eqref{cup}, we obtain an isomorphism of mixed
Hodge structures
$$\wedge^i_{\Z} H^1(X_\infty,\Z)\rightarrow H^i(X_\infty,\Z)$$
that is compatible with the action of $M$.

(2) For every $i\geq 0$, Poincar\'e duality yields a
 natural isomorphism of sheaves of $\Z$-modules
\begin{equation}\label{eq-dualVH}\alpha:R^if_*(\Z)^{\vee}\to R^{2g-i}f_*(\Z)(g)\end{equation} (note that
Poincar\'e duality holds with coefficients in $\Z$ because
$R^if_*(\Z)$ is a locally free sheaf of $\Z$-modules for all
$i\geq 0$).
 By
\cite[6.19]{peters-steenbrink}, this isomorphism respects the
Hodge structure on every fiber of $R^if_*(\Z)^{\vee}$ and
$R^{2g-i}f_*(\Z)(g)$. Thus \eqref{eq-dualVH} is an isomorphism of
variations of Hodge structures on $S$. As in (1), one checks in a
straightforward way that the limit of $R^if_*(\Z)^{\vee}$ at $s$
is the dual of the limit of $R^if_*(\Z)$, using Proposition
\ref{prop-weightlin} to verify the compatibility of the weight
filtrations. Thus the isomorphism \eqref{eq-dualVH} induces an
isomorphism of mixed Hodge structures
$$H^i(X_\infty,\Z)^{\vee}\to H_i(X_\infty,\Z).$$
\end{proof}

By Proposition \ref{prop-cup}, in order to describe the limit
mixed Hodge structure on $H_i(X_{\infty},\Q)$ and
$H^i(X_\infty,\Z)$ for all $i\geq 0$, it suffices to determine the
limit mixed Hodge structure on $H_1(X_\infty,\Z)$.

\subsection{Description of the mixed Hodge structure on
$H_1(X_\infty,\Z)$}
 We denote by
$$\mathscr{V}\rightarrow S^{\an}$$ the polarized variation of Hodge structures
$$R^{2g-1}f_*(\Z)(g)$$ of type $\{(0,-1),(-1,0)\}$
\cite[4.4.3]{HodgeII}.  We denote by $\mathscr{V}_{\Z}$,
$\mathscr{V}_{\Q}$ and $\mathscr{V}_{\C}$ the integer, resp.
rational, resp. complex component of $\mathscr{V}$. The sheaf
$\mathscr{V}_{\Z}$ is a locally free sheaf of $\Z$-modules on
$S^{\an}$ of rank $2g$.
 The fiber of $\cV$ over a point $z$ of $S^{\an}$ is
canonically isomorphic to the weight $-1$ Hodge structure
$$H^{2g-1}(X_z(\C),\Z)(g),$$
 where $X_z$ denotes the fiber of $f$ over $z$. By Poincar\'e
 duality, there is a canonical isomorphism of $\Z$-modules
 $$H^{2g-1}(X_z(\C),\Z)\cong H_1(X_z(\C),\Z).$$

The limit of $\cV$ at the point $s$ is a mixed Hodge structure
that was constructed by Schmid \cite{schmid}. In our notation,
this limit is
 precisely the mixed Hodge structure $H_1(X_\infty,\Z).$ For a quick
introduction to limit mixed Hodge structures, we refer to
\cite{hain} and \cite[\S10 and \S11]{peters-steenbrink}.
 Here we will only briefly sketch the main ingredients of the construction.

We consider a small disc $\Delta$ around $s$ in $S^{\an}$ and we
denote by $\Delta^*$ the punctured disc $\Delta\setminus \{s\}$.
  It follows from the
definition of the nearby cycles functor that $H_1(X_\infty,\Z)$ is
the $\Z$-module of global sections of the pullback of $\cV_{\Z}$
to a universal cover of $\Delta^*$ \cite[XIV.1.3.3]{sga7b}.
 By the comparison isomorphism
\eqref{comparcoh} and \cite[IX.3.5]{sga7a}, the action of the
monodromy operator $M^d$ on $H_1(X_{\infty},\Q)$ is unipotent of
level $\leq 2$ (meaning that $(M^d-\Id)^2=0$). The level of
unipotency can also be deduced from the Monodromy Theorem
\cite[6.1]{schmid} and the fact that the fibers of $\cV$ are of
type $\{(-1,0),(0,-1)\}$.

The {\em weight filtration} $W$ on $H_1(X_{\infty},\Q)$ is the
weight filtration with center $-1$ associated to the nilpotent
operator $N$ (recall that $N$ is the logarithm of the unipotent
part $M_u$ of the monodromy operator $M$). Since $M_u^d=M^d$, we
have
$$dN=\log(M_u^d)=M^d-\Id$$
so that $N^2=0$ and the weight filtration is of the form
$$\{0\}\subset W_{-2}\subset W_{-1}\subset W_0=H_1(X_{\infty},\Q).$$
Explicitly, we have $W_{-1}=\ker(N)$ and $W_{-2}=\mathrm{im}(N)$.
This filtration induces a weight filtration on the lattice
$H_1(X_\infty,\Z)$ in $H_1(X_\infty,\Q)$. Note that
$$W_{-1}H_1(X_\infty,\Z):=W_{-1}H_1(X_\infty,\Q)\cap
H_1(X_\infty,\Z)$$ is stable under the action of $M_s$, since
$H_1(X_\infty,\Z)$ is stable under $M$ and $M$ is semi-simple on
$W_{-1}H_1(X_\infty,\Q)=\ker(M^d-\Id)$.

 Consider the finite covering
$$\widetilde{\Delta}\to \Delta$$ obtained by taking a $d$-th root $t'$ of the coordinate $t$ on $\Delta$. This covering is totally ramified
over the origin $s$ of $\Delta$.  With a slight abuse of notation,
we denote again by $s$ the unique point of the open disc
$\widetilde{\Delta}$ that lies over $s\in \Delta$. We denote by
$\widetilde{\Delta}^*$ the punctured disc
$\widetilde{\Delta}\setminus \{s\}$.

Pulling back $\cV$ to a variation of Hodge structures $\cV'$ on
$\widetilde{\Delta}^*$ has the effect of raising the monodromy
operator $M$ to the power $d$. This has no influence on the
associated weight filtration on $H_1(X_{\infty},\Q)$, since the
logarithm of $(M_u)^d=M^d$ equals $dN$. Pulling back $\cV$ to
$\widetilde{\Delta}^*$ is the first step in the construction of
the limit Hodge filtration $F$ on $H_1(X_{\infty},\Q)$. The
important point is that the monodromy $M^d$ of the variation
$\cV'$ is unipotent.

Schmid considers a complex manifold $\check{D}$ that parameterizes
descending filtrations
$$F^1=H_1(X_\infty,\C)\supset F^0\supset \{0\}$$
 on $H_1(X_{\infty},\C)$ that satisfy a certain compatibility relation with the polarization
 on $H_1(X_{\infty},\C)$ and such that $F^0$ has dimension $g$.
 ``Untwisting'' the monodromy action on the fibers of $\cV'$, he
 constructs a map
 $$\widetilde{\Psi}:\widetilde{\Delta}^*\to \check{D}.$$
 The Nilpotent Orbit
 Theorem (for one variable) in \cite[4.9]{schmid} guarantees that $\widetilde{\Psi}$
 extends to a holomorphic map  $$\Psi:\widetilde{\Delta}\to
 \check{D}.$$ The point $\Psi(s)$ of $\check{D}$ corresponds to a descending filtration
 $$F^{-1}=H_1(X_\infty,\C)\supset F^{0}\supset \{0\}$$ that is called
 the {\em limit Hodge filtration} on $H_1(X_\infty,\C)$.
 Schmid's
 fundamental result \cite[6.16]{schmid} states that the weight
 filtration $W$ and the limit Hodge filtration $F$
 define a polarized mixed Hodge structure on $H_1(X_\infty,\Z)$, which is
 called the {\em limit mixed Hodge structure} of the variation of
 Hodge structures $\cV'$.

We see from the shape of the weight filtration and the limit Hodge
filtration that the limit mixed Hodge structure on
$H_1(X_\infty,\Z)$ is of type
$$\{(0,0),\,(-1,0),\,(0,-1),\,(-1,-1)\}.$$  Moreover, since
$$\Gr_{-1}^WH_1(X_\infty,\Z)$$ is polarizable, the mixed Hodge
structure
 $$(H_1(X_\infty,\Z),W, F)$$ is a mixed Hodge $1$-motive in the sense of
 \cite[\S\,10]{HodgeIII}.

The construction of the limit Hodge filtration can be reformulated
in terms of the {\em Deligne extension} or {\em canonical
extension}. Consider the holomorphic vector bundle
$$(\cV')^h=\cV'_{\Z}
\otimes_{\Z}\mathcal{O}_{\widetilde{\Delta}^*}$$ on the punctured
disc $\widetilde{\Delta}^*$. The locally constant subsheaf
$\cV'_{\C}$ of this vector bundle defines a connection $\nabla$ on
$(\cV')^h$, called the Gauss-Manin connection. The vector bundle
$(\cV')^h$ extends in a unique way to a vector bundle
$\widehat{\cV'}$ on $\widetilde{\Delta}$ such that $\nabla$
extends to a logarithmic connection on $\widehat{\cV'}$ whose
residue at $s$ is nilpotent \cite[5.2]{deligne-diffeq}. We call
$\widehat{\cV'}$ the {\em Deligne extension} of the variation of
Hodge structures $\cV'$. The fiber of the Deligne extension over
the origin $s$ of $\widetilde{\Delta}$ can be identified with
$H_1(X_\infty,\C)$, by \cite[XI-8]{peters-steenbrink}
 (this
identification depends on the choice of a local coordinate on
$\widetilde{\Delta}$; we take the coordinate $t'$).

The Hodge flags $F^0$ on the fibers of $\cV'$ glue to a
holomorphic subbundle $F^0(\cV')^h$ of $(\cV')^h$, which extends
uniquely to a holomorphic subbundle $F^0\widehat{\cV'}$ of
$\widehat{\cV'}$ \cite[11.10]{peters-steenbrink}. Taking fibers at
$s$, we obtain a descending filtration
$$(\widehat{\cV'})_s=H_1(X_\infty,\C)\supset
(F^0\widehat{\cV'})_s\supset \{0\}$$ and this is precisely the
limit Hodge filtration on $H_1(X_\infty,\C)$
\cite[11.10]{peters-steenbrink}.

Schmid's construction of the limit mixed Hodge structure works for
abstract variations of Hodge structures, which need not
necessarily come from geometry. If the variation of Hodge
structures comes from the cohomology of a proper and smooth family
over $S$ (such as in the case that we are considering), the
Deligne extension and the extension of the Hodge bundels can be
constructed explicitly using a relative logarithmic de Rham
complex associated to a suitable compactification $\overline{f}$.
This is Steenbrink's construction \cite{steenbrink-limit}. We will
not need this approach in this paper.

\begin{theorem}\label{theo-mhs}
We apply the terminology of Section \ref{subsec-neron} to the
abelian $K$-variety $A$ and define in this way the degree $d$
extension $K'$ of $K$, as well as the torus $T$ and the abelian
variety $B$ over $\C$, endowed with a right action of $\mu\cong
\mu_d(\C)$. There exist canonical $\mu$-equivariant isomorphisms
of pure Hodge structures
\begin{eqnarray*}
\Gr_0^W(H_1(X_\infty,\Q))&\cong& H_1(T(\C),\Q)(-1)
\\ \Gr_{-1}^W(H_1(X_\infty,\Z))&\cong & H_1(B(\C),\Z)
\\ \Gr_{-2}^W(H_1(X_\infty,\Z))&\cong & H_1(T(\C),\Z).
\end{eqnarray*}
\end{theorem}
\begin{proof}
We denote by $\C(S')$ the algebraic closure of the function field
$\C(S)$ in $K'$, and we consider the normalization
$$\overline{S}'\rightarrow \overline{S}$$
 of $\overline{S}$ in $\C(S')$. This is a ramified Galois covering, obtained by taking a $d$-th root of the local parameter $t$. Its Galois group is canonically
 isomorphic to $\mu$. With abuse of notation, we denote again by $s$ the unique point of the inverse image of $s$ in
$\overline{S}'$, and we put $S'=\overline{S}'\setminus \{s\}$.
 Then
$$f':X'=X\times_{S}S'\rightarrow S'$$ is a smooth projective
family of abelian varieties, and we have a canonical isomorphism
$$A'=A\times_K K'\cong X'\times_{S'}\Spec K'.$$
As we've argued above, the fact that $A'$ has semi-abelian
reduction implies that the variation of Hodge structures
$$\mathscr{V}'=\mathscr{V}\times_{S}S'\cong R^{2g-1}f'_*(\Z)(g) $$
 has unipotent monodromy around $s$.

We denote by $\mathcal{X}'$ the N\'eron model of $X'$ over
$\overline{S}'$, and by $\mathcal{A}'$ the N\'eron model of $A'$
over $R'$, where $R'$ denotes the integral closure of $R$ in $K'$.
Note that there is a canonical isomorphism of $R'$-schemes
$$\mathcal{A}'\cong \mathcal{X}'\times_{\overline{S}'}\Spec R'.$$
 The analytic family of abelian
varieties
$$(f')^{\an}:(X')^{\an}\rightarrow (S')^{\an}$$ is canonically isomorphic to
the Jacobian
$$J(\mathscr{V}')\rightarrow (S')^{\an}$$ of the variation of Hodge
structures $\mathscr{V}'$ \cite[2.10.1]{saito-adm}. We will now
explain the relation between the complex semi-abelian variety
$(\mathcal{A}')^o_s$ and the limit mixed Hodge structure
$H_1(X_{\infty},\Z)$ of $\mathscr{V}'$ at the point $s$. To
simplify notation, we put $H_C=H_1(X_\infty,C)$ for
$C=\Z,\,\Q,\,\C$,
and  we denote by $H$ the mixed Hodge structure
$$(H_{\Z},W_{\bullet}H_{\Q},F^{\bullet}H_{\C}).$$

By \cite[4.5(i)]{saito-adm}, $(\mathcal{X}')^{\an}$ is canonically
isomorphic to Clemens' N\'eron model of $\mathscr{V}'$ over
$\overline{S}'$; see \cite{clemens-neron} and
\cite[2.5]{saito-adm} for a definition. In \cite[2.5]{saito-adm},
Clemens' N\'eron model is constructed by gluing copies of  the
Zucker extension $J^{Z}_{\overline{S}'}(\mathscr{V}')$ of
$\mathscr{V}'$, which is defined in \cite{zucker-neron} and
\cite[2.1]{saito-adm}.
 It follows immediately from
this construction that the identity component
$$((\mathcal{X}')^o)^{\an}$$ of Clemens' N\'eron model is canonically isomorphic
to the Zucker extension $J^{Z}_{\overline{S}'}(\mathscr{V}')$.

The Zucker extension is given explicitly by
$$J^{Z}_{\overline{S}'}(\mathscr{V}')=j_*\mathscr{V}'_{\Z}\backslash \widehat{\mathscr{V}'}/F^0\widehat{\mathscr{V}'}$$
where $\widehat{\mathscr{V}'}$ is the Deligne extension of
$\mathscr{V}'$ to $\overline{S}'$, $j$ is the open immersion of
$S'$ into $\overline{S}'$, and $F^0\widehat{\mathscr{V}'}$ is the
unique extension of the holomorphic vector bundle
$$F^0(\mathscr{V}'_{\Z}\otimes_{\Z}\mathcal{O}_{(S')^{\an}})$$ to a
holomorphic subbundle of $\widehat{\mathscr{V}'}$. We can describe
the fiber
$$J^{Z}_{\overline{S}'}(\mathscr{V}')_s\cong
((\mathcal{X}')^o_s)^{\an}\cong ((\mathcal{A}')^o_s)^{\an}$$  of
$J^{Z}_{\overline{S}'}(\mathscr{V}')$ at $s$ in terms of the mixed
Hodge structure $H$, as follows.

As we've explained above, the fiber of $\widehat{\mathscr{V}'}$
over $s$ is isomorphic to $H_{\C}$, and
$F^0(\widehat{\mathscr{V}'})_s$ coincides with the degree zero
part of the Hodge filtration on $H_{\C}$. Moreover, the fiber of
$j_*\mathscr{V}'_{\Z}$ at $s$ is the $\Z$-module of elements in
$H_{\Z}$ that are invariant under the monodromy action of $M^d$.
By definition, the weight filtration on $H_{\Q}$ is the filtration
centered at $-1$ defined by the logarithm $N$ of the unipotent
part $M_u$ of $M$, or, equivalently, the logarithm $N'=dN$ of
$M^d$. Since $(M^d-\Id)^2=0$, we have $N'=M^d-\Id$ and $(N')^2=0$,
and we see that
$$(j_*\mathscr{V}'_{\Z})_s=\ker(N')=W_{-1}H_{\Z}.$$
Thus, we find canonical isomorphisms
\begin{eqnarray*}
((\mathcal{A}')^o_s)^{\an}&\cong &
J^{Z}_{\overline{S}'}(\mathscr{V}')_{s} \\&\cong &
W_{-1}H_{\Z}\backslash H_{\C}/F^0 H_{\C}\\&\cong&
W_{-1}H_{\Z}\backslash W_{-1}H_{\C}/(F^0H_{\C}\cap
W_{-1}H_{\C}).\end{eqnarray*} The last isomorphism is obtained as
follows: since $\Gr_0^WH$ is purely of type $(0,0)$, we have
$$F^0\Gr_0^WH_{\C}=\Gr_0^WH_{\C}=H_{\C}/W_{-1}H_{\C},$$ so that the
morphism
$$W_{-1}H_{\C}\to H_{\C}/F^0 H_{\C}$$ induced by the inclusion of
$W_{-1}H_{\C}$ in $H_{\C}$ is surjective. Its kernel is precisely
$F^0H_{\C}\cap W_{-1}H_{\C}$.

 By \cite[10.1]{HodgeIII}, we have an extension
\begin{equation}\label{eq-ext}
0\rightarrow J(\Gr^W_{-2}H) \rightarrow
((\mathcal{A}')^o_s)^{\an}\rightarrow J(\Gr^W_{-1}H)\rightarrow
0\end{equation} where
$$J(\Gr^W_{-2}H)=\Gr^W_{-2}H_{\C}/\Gr^W_{-2}H_{\Z}$$
 is a torus, and
$$J(\Gr^W_{-1}H)=H_{\Z}\backslash \Gr^W_{-1}H_{\C}/F^0
\Gr^W_{-1}H_{\C}$$
  an abelian variety, because the Hodge structure $\Gr^W_{-1}H$ is polarizable.
  By  \cite[10.1.3.3]{HodgeIII}, the extension (\ref{eq-ext}) must
  be
the analytification of the Chevalley decomposition
$$0\rightarrow T\rightarrow (\mathcal{A}')_s^o \rightarrow
B\rightarrow 0.$$ Hence, there exist canonical isomorphisms
 of pure Hodge structures
\begin{eqnarray}
\Gr_{-1}^W(H)&\cong & H_1(B(\C),\Z), \label{eq-iso1}
\\ \Gr_{-2}^W(H)&\cong & H_1(T(\C),\Z). \label{eq-iso2}
\end{eqnarray}
Moreover, by definition of the weight filtration on $H_{\Q}$, the
operator $N$ defines a $\mu$-equivariant isomorphism of $\Q$-Hodge
structures
$$\Gr_{0}^W(H)\otimes_{\Z}\Q\rightarrow
\Gr_{-2}^W(H)(-1)\otimes_{\Z}\Q.$$
It remains to show that the isomorphisms (\ref{eq-iso1}) and
(\ref{eq-iso2}) are $\mu$-equivariant. It is enough to prove that
the Galois action of $\mu$ on
$$\mathscr{V}'\rightarrow S'$$
 extends analytically to the Zucker extension
$$J^{Z}_{\overline{S}'}(\mathscr{V}')\rightarrow \overline{S}'$$
in such a way that the action of the canonical generator of
$\mu=\mu_d(\C)$ on
$$J^{Z}_{\overline{S}'}(\mathscr{V}')_s=W_{-1}H_{\Z}\backslash H_{\C}/F^0
H_{\C}$$ coincides with the semi-simple part $M_s$ of the
monodromy action. This follows easily from the constructions.
\end{proof}

\subsection{Multiplicity functions and limit mixed Hodge
structure}


\begin{theorem}\label{thm-weight}
We keep the notations of Section \ref{subsec-limMHS}.
\begin{enumerate}
\item The potential toric rank $t_{\pot}(A)$ is equal to the
largest integer $\alpha$ such that $\exp(2\pi c(A)i)$ is an
eigenvalue of $M_s$ on $\Gr_{g+\alpha}^WH^g(X_{\infty},\Q)$.
 \item The Jordan form of $M_s$ on
$$\begin{array}{lcl}
 (\Gr^W_{-1}H_1(X_{\infty},\Q))^{1,0} &\mbox{ is }&
 \Jord(m_A^{\ab},0),\\[1.5ex]
(\Gr^W_{-1}H_1(X_{\infty},\Q))^{0,1} &\mbox{ is }&
\Jord(\dualm_A^{\ab},0),\\[1.5ex] \Gr^W_{-2}H_1(X_{\infty},\Q) &\mbox{ is
}& \Jord(m^{\tor}_A,0),
\\[1.5ex] \Gr^W_{0}H_1(X_{\infty},\Q) &\mbox{ is
}& \Jord(m^{\tor}_A,0).
 \end{array}$$
\end{enumerate}
\end{theorem}

\begin{proof}
We denote by $M_u$ the unipotent part of the monodromy operator
$M$, and by $N$ its logarithm. By definition, the weight
filtration on $H^g(X_\infty,\Q)$ is the filtration with center $g$
associated to the nilpotent operator $N$.
%
 Thus (1) follows from Proposition \ref{prop-weightjord}, Theorem \ref{thm-blocksize} and the isomorphism
(\ref{comparcoh}).

Point (2) follows from Theorem \ref{theo-mhs} and the canonical
$\mu$-equivariant isomorphisms $$\begin{array}{lll}
H_{1}(B(\C),\C)^{1,0}&\cong& \Lie(B)
\\ H_1(B(\C),\C)^{0,1}&\cong& \Lie(B^{\vee})^{\vee}
\\ H_1(T(\C),\C)&\cong & \Lie(T)
\end{array}$$
(see \cite[pp. 4 and 86]{mumford-AV} for the dual isomorphisms on
the level of cohomology).
\end{proof}


\begin{thebibliography}{1}

\bibitem{sga3.1}
{\em Sch\'emas en groupes. {I}: {P}ropri\'et\'es g\'en\'erales des
sch\'emas en
  groupes}.
\newblock S\'eminaire de G\'eom\'etrie Alg\'ebrique du Bois Marie 1962/64 (SGA
  3). Dirig\'e par M. Demazure et A. Grothendieck. Lecture Notes in
  Mathematics, Vol. 151. Springer-Verlag, Berlin, 1970.


  \bibitem{sga7a}
{\em Groupes de monodromie en g\'eom\'etrie alg\'ebrique. {I}}.
\newblock S\'eminaire de G\'eom\'etrie Alg\'ebrique du Bois-Marie 1967--1969
  (SGA 7 {I}), Dirig\'e par A. Grothendieck. Avec la collaboration de M.
  Raynaud et D. S. Rim. Lecture Notes in Mathematics, Vol. 288.
\newblock Springer-Verlag, Berlin, 1972.

\bibitem{sga7b}
{\em Groupes de monodromie en g\'eom\'etrie alg\'ebrique. {II}},
 \newblock S\'eminaire de G\'eom\'etrie Alg\'ebrique du Bois-Marie
              1967--1969 (SGA 7 II),
              Dirig\'e par P. Deligne et N. Katz.
              Lecture Notes in Mathematics, Vol. 340.
 \newblock Springer-Verlag, Berlin, 1973.


\bibitem{neron}
S.~Bosch, W.~{L\"u}tkebohmert, and M.~Raynaud.
\newblock {\em {N\'eron models}}, volume~21 of
 {\em Ergebnisse der Mathematik und ihrer Grenzgebiete}. \newblock Springer-Verlag, 1990.

\bibitem{chai}
C.-L. Chai.
\newblock {N\'eron models for semiabelian varieties: congruence and change of
  base field}.
\newblock {\em Asian J. Math.}, 4(4):715--736, 2000.

\bibitem{clemens-neron}
H.~Clemens.
\newblock The N\'eron model for families of intermediate Jacobians acquiring ``algebraic'' singularities.
\newblock{\em Inst. Hautes \'Etudes Sci. Publ. Math.} 58:5--18, 1983.


\bibitem{deligne-diffeq}
P.~Deligne.
\newblock {\em \'{E}quations diff\'erentielles \`a points singuliers
              r\'eguliers}, volume 163 of
{\em Lecture Notes in Mathematics}.
 Berlin: Springer-Verlag, 1970.


\bibitem{HodgeII}
P.~Deligne.
\newblock{Th\'eorie de Hodge: II.}
\newblock{\em Inst. Hautes \'Etudes Sci. Publ. Math.}, 40:5--57,
1971.



\bibitem{HodgeIII}
P.~Deligne.
\newblock{Th\'eorie de Hodge: III.}
\newblock{\em Inst. Hautes \'Etudes Sci. Publ. Math.}, 44:5--77,
1974.

 \bibitem{deligne-weilII}P.~Deligne.
\newblock La conjecture de {W}eil: {II}.
\newblock{\em Inst. Hautes \'Etudes Sci. Publ. Math.}, 52:137–-252,
1980.

\bibitem{dimca}
A.~Dimca. \newblock{\em Sheaves in topology.}
 Universitext.
Springer-Verlag, Berlin, 2004.

\bibitem{edix}
B.~Edixhoven.
\newblock N\'eron models and tame ramification.
\newblock {\em Compos. Math.}, 81:291--306, 1992.

\bibitem{hain}
R.~Hain. \newblock Periods of limit mixed Hodge structures.
\newblock In: {\em Current developments in mathematics}, pages 113-–133, Int. Press,
Somerville, MA, 2003.

\bibitem{HaNi}
L.H.~ Halle and J.~Nicaise.
\newblock{Motivic zeta functions of abelian varieties, and the
monodromy conjecture.} \newblock {\em Adv. Math.}, 227:610--653,
2011.


\bibitem{kashi2}
M.~Kashiwara.
\newblock {Holonomic systems of linear differential equations with regular
  singularities and related topics in topology.}
\newblock In: {\em {Algebraic varieties and analytic varieties, Proc. Symp.,
  Tokyo 1981}}, volume~1 of {\em Adv. Stud. Pure Math.},  pages 49--54, North-Holland, Amsterdam-New York, 1983.

\bibitem{Loepadic}
F.~Loeser.
\newblock Fonctions d'{I}gusa p-adiques et polyn\^omes de {B}ernstein.
\newblock {\em Am. J. of Math.}, 110:1--22, 1988.

\bibitem{Mal2}
B.~Malgrange.
\newblock {Polyn\^{o}mes de Bernstein-Sato et cohomologie \'evanescente}.
\newblock {\em Ast\'erisque}, 101/102:243--267, 1983.



\bibitem{milne-abelian}
J.S.~Milne.
\newblock {\em Abelian varieties.}
\newblock In: G. Cornell and J.H. Silverman (eds.), {\em
Arithmetic geometry.} Springer Verlag, 1986.

\bibitem{mumford-AV}
D.~Mumford.
\newblock {\em {Abelian varieties. With appendices by C. P. Ramanujam and Yuri
  Manin. 2nd ed.}}, volume~5 of {\em {Tata Institute of Fundamental Research
  Studies in Mathematics}}.
\newblock {London: Oxford University Press}, 1974.




\bibitem{Ni-japan}
J.~Nicaise.
\newblock An introduction to $p$-adic and motivic zeta functions and the monodromy conjecture.
\newblock In: G. Bhowmik, K. Matsumoto and H. Tsumura (eds.), {\em Algebraic and analytic aspects of zeta functions and L-functions}, volume 21 of {\em MSJ Memoirs},
Mathematical Society of Japan, pages 115-140, 2010.

\bibitem{oda}
T.~Oda.
\newblock The first de Rham cohomology group and Dieudonn{\'e} modules.
\newblock {\em Ann. Sci. {\'E}cole Norm. Sup. (4)}, 2(1):63-135, 1969.

\bibitem{peters-steenbrink}
C.~Peters and J.H.M. Steenbrink.
\newblock {\em {Mixed Hodge structures}}, volume~52 of {\em {Ergebnisse der
  Mathematik und ihrer Grenzgebiete. 3. Folge}}.
\newblock Berlin: Springer, 2008.

\bibitem{Rod}
B.~Rodrigues.
\newblock {On the monodromy conjecture for curves on normal surfaces.}
\newblock {\em Math. Proc. Camb. Philos. Soc.}, 136(2):313--324, 2004.

\bibitem{saito-adm}
M.~Saito.
\newblock Admissible normal functions.
\newblock{\em J. Algebraic Geom.} 5, 235--276, 1996.

\bibitem{schmid}
W.~Schmid. \newblock{Variation of Hodge structure: the
singularities of the period mapping.}\newblock {\em Invent. Math.}
22:211-–319, 1973.

\bibitem{steenbrink-limit}
J.H.M.~Steenbrink.
\newblock {Limits of Hodge structures.}
\newblock {\em Invent. Math.}, 31:229--257, 1976.

\bibitem{steenbrink-vanish}
J.H.M.~Steenbrink. \newblock { Mixed {H}odge structure on the
vanishing cohomology}. In: {\em Real and complex singularities
({P}roc. {N}inth {N}ordic
              {S}ummer {S}chool/{NAVF} {S}ympos. {M}ath., {O}slo,
              1976)}, pages 525--563. {Alphen aan den Rijn: Sijthoff and Noordhoff}, 1977.

\bibitem{oort-tate}
J.~Tate and F.~Oort. \newblock {Group schemes of prime order.}
\newblock{\em Ann. Sci. {\'E}cole Norm. Sup.} (4) 3:1-–21, 1970.



\bibitem{zucker-neron}
S.~Zucker. \newblock{Generalized intermediate Jacobians and the
theorem on normal functions.} \newblock{\em Invent. Math.}
33(3):185--222, 1976.


\end{thebibliography}
\end{document}